%% file: Cheeger.tex
\documentclass{amsart}

\usepackage{amsfonts,amssymb,verbatim,amsmath,amsthm,latexsym,textcomp,amscd}
\usepackage{latexsym,amsfonts,amssymb,epsfig,verbatim}
\usepackage{amsmath,amsthm,amssymb,latexsym,graphics,textcomp, hyperref}
\usepackage{graphicx}
\usepackage{color}
\usepackage{url}
\usepackage{psfrag}
\usepackage{amsmath,comment,amscd}
\usepackage{amssymb}
\usepackage{amsthm}

\usepackage{caption}
\usepackage{color}
\usepackage{enumerate}
\usepackage{float}
\usepackage[a4paper]{geometry}
\usepackage{graphicx}

\numberwithin{equation}{section}

\newcommand{\ds}{\displaystyle}
\newcommand{\define}{\textbf}

\newcommand{\lap}{\Delta}
\newcommand{\mc}{\mathcal}

\newcommand{\constone}{{\textbf{1}}}
\newcommand{\ra}{\Rightarrow}
\newcommand{\vp}{\varphi}

\newcommand{\ab}[1]{\langle#1\rangle} 

\newcommand{\lrb}[1]{\left[#1\right]}

\newcommand{\lrp}[1]{\left(#1\right)}
\newcommand{\lrmod}[1]{\left|#1\right|}

\newcommand{\lrset}[1]{\left\{#1\right\}}
\newcommand{\norm}[1]{\|#1\|}
\newcommand{\set}[1]{\{#1\}}
\newcommand{\suppressthis}[1]{}

\newcommand{\reason}[2]{\stackrel{\text{#1}}{#2}}

\newcommand{\twopartdef}[4]
{
	\left\{
	\begin{array}{ll}
		#1 & \mbox{if } #2 \\
		#3 & \mbox{if } #4
	\end{array}
	\right.
}

\newcommand{\R}{\mathbf R}

\newcommand{\Z}{\mathbf Z}

\DeclareMathOperator{\im}{Im}
\DeclareMathOperator{\vol}{vol}

\newtheorem{theorem}{Theorem}[section]
\newtheorem{lemma}[theorem]{Lemma}

\newtheorem{example}[theorem]{Example}
\newtheorem{prop}[theorem]{Proposition}

\newtheorem{defn}[theorem]{Definition}
\newtheorem{definition}[theorem]{Definition}

\newtheorem{rmk}[theorem]{Remark}

\newcommand\FF{{\mathcal F}}

\newcommand\LL{{\mathcal L}}
\newcommand\MM{{\mathcal M}}

\newcommand\PP{{\mathcal P}}

\newcommand\PMF{{\PP\kern-2pt\MM\FF}}

\newcommand\PML{{\PP\kern-2pt\MM\LL}}

\newcommand\E{{\mathbb E}}

\DeclareMathOperator{\var}{Var}
\newcommand{\one}{\mathbf 1}

\begin{document}

\title{Cheeger inequalities for graph limits}

\author{Abhishek Khetan}
\address{School
	of Mathematics, Tata Institute of Fundamental Research. 1, Homi Bhabha Road, Mumbai-400005, India}

\email{khetan@math.tifr.res.in}

\author{Mahan Mj}
\address{School
	of Mathematics, Tata Institute of Fundamental Research. 1, Homi Bhabha Road, Mumbai-400005, India}

\email{mahan@math.tifr.res.in}
\email{mahan.mj@gmail.com}

\thanks{ Research of second author partly supported by a DST JC Bose Fellowship and partly by a MATRICS research project (FILE NO. MTR/2017/000005).}
\date{\today}

\begin{abstract}
We introduce  notions of Cheeger constants for graphons and graphings. We prove Cheeger and Buser inequalities for these. On the way we prove co-area formulae for graphons and graphings.
\end{abstract}

\maketitle

\tableofcontents
\section{Introduction} The Cheeger constant, introduced in Riemannian geometry by Cheeger \cite{cheeger-art} in the early 70's measures the `most efficient' way to cut a closed Riemannian manifold into two pieces, where efficiency is measured in terms of an isoperimetric constant.
Cheeger \cite{cheeger-art} and Buser related this geometric quantity to a spectral quantity--the bottom of the spectrum of the Laplacian. These are the well-known Cheeger-Buser inequalities in Riemannian geometry (see \cite[Section 8.3]{buser-book} for details).
A discrete version of the Cheeger constant and the Cheeger-Buser inequalities  was then obtained independently by Dodziuk \cite{dodziuk-cheeger} and Alon-Milman \cite{alon-cheeger,alon-milman} for finite graphs (see \cite{chung-cheeger} for a number of different proofs and \cite{mohar-survey1} for a survey). These ideas and inequalities have also been extended to weighted  graphs \cite{friedland_nabben_weighted_cheeger} (see also \cite[Chapter 2, pg. 24]{fan_chung_spectral_graph_theory}, \cite{trevisan}). In a certain sense, this marked a fertile way of discretizing a notion that arose in the setup of continuous geometry.

More recently,
the theory of graph limits, graphons and graphings was developed by Lovasz \cite{lovasz_large_networks} and others (see especially \cite{bclsv1,bclsv2,bckl,chatterjee-lnm}) giving a method of obtaining measured continua from infinite sequences of finite graphs. From a certain point of view, this gives us a path in the opposite direction: from the discrete to the continuous.

Such continuous limits come in two flavors: dense graphs (graphons) or
sparse graphs (graphings).
A graphon is relatively easy to describe: it is a bounded (Lebesgue) measurable function $W:I^2\to I$ that is symmetric: $W(x, y)=W(y, x)$ for all $x, y\in I$. A graphing on the other hand may be thought of as a measure on $I^2$ that can be  locally described as a product  of a sub-probability measure on $I$ with the counting measure on a set of uniformly bounded cardinality (see Sections \ref{prel-graphon} and \ref{prel-graphing} for details). Each of the co-ordinate intervals $I\times \{0\}$ and $\{0\}\times I$ may be thought of as the vertex set of the graphon or graphing and is equipped with a Borel measure.

The aim of this paper is to define the notion of a Cheeger constant for graphons and graphings and prove the Cheeger-Buser inequalities for them.
For both a graphon $W$ and a graphing $G$,  the Cheeger constants $h_W$ and $h_G$ respectively measure (as in Cheeger's original definition) the best way to partition the ``vertex set" $I$ into $A, A^c$ such that  the isoperimetric constant is minimized. For instance for a graphon $W$,
\begin{equation}
h_W= \inf_{A\subseteq I:\ 0< \mu_L(A)< 1}\frac{e_W(A, A^c)}{\min\set{\vol_W(A), \vol_W(A^c)}},
\end{equation}
where ${e_W(A, A^c)}$ measures the total measure of edges between $A, A^c$  (see Definitions \ref{def:cheeger-graphon} and \ref{def-graphing} below for details). A rather different   Cheeger-type inequality for graphings (but not for graphons) involving von Neumann algebras was explored by Elek in
\cite{elek-cheeger}.

The main theorem of the paper  is the following
(see Theorems \ref{buser-graphon}, \ref{cheeger-graphon}, \ref{buser-graphing} and \ref{cheeger-graphing}):

\begin{theorem}\label{intro}
	Let $W$ be a connected graphon and $\lambda_W$ denote the  bottom of the spectrum of the Laplacian.
	Then
	\begin{equation*}
	\frac{h^2_W}{8} \leq	\lambda_W\leq 2h_W.\\
	\end{equation*}

	Again, let $G$ be a connected graphing and $\lambda_G$ denote the bottom of the spectrum of the Laplacian.
	Then
	\begin{equation*}
	\frac{h^2_G}{8} \leq	\lambda_G\leq 2h_G.
	\end{equation*}
\end{theorem}

Connectedness in the hypothesis of Theorem \ref{intro} above is a mild technical restriction to ensure that the Cheeger constant is well-defined.\\

\noindent{\bf Finite graphs versus graph limits:} The classical Cheeger-Buser inequalities for  finite graphs can be obtained (modulo a factor of 4) as an immediate consequence of Theorem \ref{intro} for graphings using the following canonical graphing that corresponds to a finite  graph.
For any finite connected graph $F$ on $\set{1 , \ldots, n}$ as the vertex set, we can define a graphing $G=(I, E, \mu)$ as follows: Let $v_i=(2i-1)/2n$ for $1\leq i\leq n$.
Define $E$ as
$$
	E=\set{(v_i, v_j):\ \set{i, j} \text{ is an edge in } F}
$$
Define $\mu$ as $\mu(v_i)=1/n$ for each $i$.
Thus $\mu(B)=0$ for all Borel sets which do not contain any of the $v_i$'s.
It is easy to check that $G$ is connected, and the Cheeger constant of $G$ and the Cheeger constant for $F$ are equal.
The same is true for $\lambda_G$ and $\lambda_F$.
So we get
$$
	\frac{h^2_F}{8}\leq \lambda_F\leq 2 h_F
$$
as a special case of Theorem \ref{intro} for graphings.

However, the situation becomes more interesting when we use graphons rather than graphings in the above. 
Indeed, any graph $F$ on $\set{1 , \ldots, n}$ naturally gives rise to a graphon $W$ by writing $W(x, y)=1$ if there is an edge between the vertices $\lceil nx\rceil $ and $\lceil ny\rceil$, and $0$ otherwise.
Clearly, $h_W\leq h_F$.
So a natural question is to ask for lower bounds on $h_W/h_F$.
We investigate this in Section \ref{section:cheeger constant of a graph versus a graphon}.
In particular, we obtain the result that if $F$ is any regular connected graph on $n$ vertices, and $W$ is the graphon arising from $F$, then
\begin{equation}
	h_W/h_F \geq (1-\varepsilon)\lrp{1- \frac{2}{n\varepsilon^2}}
\end{equation}
for all $0<\varepsilon<1$.
So if $n$ is large then the two Cheeger constants are close.
The proof of this assertion is probabilistic in nature.

In Section \ref{section:comparing the bottom of the spectrum} we see that the smallest nonzero eigenvalue of the (normalized) Laplacian of $F$ is same as the bottom of the spectrum of the Laplacian for $W$.
So for regular graphs one can recover Cheeger type inequalities from the Cheeger inequalities for graphons.\\

\noindent{\bf Formalism of differential forms:}
We have stated Theorem \ref{intro}  in the form above to demonstrate the fact that the statements for graphons and graphings are essentially identical. In fact, once the preliminaries about graphons and graphings are dealt with in Section \ref{prel}, the proof of Theorem \ref{intro} in the two cases of graphons and graphings follows essentially the same formal route.
Thus, though structurally graphons and graphings are quite dissimilar, the proofs of the Cheeger-Buser inequalities have striking parallels.
This is quite unlike some of the other spectral theorems exposed in \cite{lovasz_large_networks} (see in particular the differences in approach in \cite{bclsv1,bclsv2,bckl}).

 To emphasize this formal similarity of proof-strategy in the two cases, Sections \ref{graphon} and \ref{graphing} have been structured in an identical manner. In both cases,
we use the formalism and language of differential forms and define the Laplacian $\Delta = d^\ast d$ on functions after proving that the ``exterior derivative" $d$ is continuous. This is adequate for Buser's inequality (Theorems \ref{buser-graphon} and \ref{buser-graphing}).
The proof of the Cheeger's inequality part of Theorem \ref{intro} we then furnish (Theorems \ref{cheeger-graphon} and \ref{cheeger-graphing})  adapts   Cheeger's original idea from \cite{cheeger-art}. Thus, we prove {\bf co-area formulae} in the two settings of graphons and graphings (see Theorems \ref{coarea-graphon} and \ref{coarea-graphing}). This might be of independent interest.\\

\noindent{\bf Connectivity:}
Finally, in Section \ref{section:cheeger constant and connectedness} we investigate  connectivity.
For a finite graph $F$ it is clear that the Cheeger constant of $F$ is positive if and only if $F$ is connected. This is equivalent, via the Cheeger-Buser inequality for finite graphs, to the statement that a graph is connected if and only if the normalized Laplacian has a one dimensional eigenspace corresponding to the zero (lowest) eigenvalue.
The analogous statement is not true for either graphons or graphings. We furnish counterexamples  in Sections \ref{sec:connectedzerocheegergraphon} and  \ref{subsection:connected graphing with vanishing cheeger constant} respectively.
However, for graphons whose degree is bounded away from zero, we prove the following equivalence (see Proposition \ref{connectedpositivecheeger}):

\begin{prop}
	Let $\varepsilon>0$ and $W$ be a graphon such that $d_W(x) \geq \varepsilon$ for all $x \in I$.
	Then $W$ is connected if and only if $h_W > 0$.
\end{prop}

We provide two proofs of this theorem, one of which uses the Cheeger-Buser inequality for graphons from Theorem \ref{intro} and the other  a structural lemma about connected graphons proved in \cite{bbcr2010}.

\subsection{Proof strategy and relation with existing literature} It is natural to try to prove the Cheeger-Buser inequalities for graphons by approximating a given graphon $W$ by a sequence of finite graphs in the cut norm and then prove that the Cheeger constants of the sequence of graphs converge to the Cheeger constant of the graphon.
However, the convergence of Cheeger constants turns out to be a subtle issue, and in general it is not true that convergence in the cut norm implies convergence of Cheeger constants (See Section \ref{subsection:convergence of cheeger constants} for a counterexample).

Thus to prove the desired inequalities we resort to techniques motivated and informed by geometry and differential topology rather than combinatorial methods. We go back to Cheeger's original proof in the context of Riemannian geometry \cite[Theorem 3, Chapter IV]{chavel}, which we
 outline for completeness:
Let $M$ be a compact Riemannian manifold without boundary.
Let $dV$ denote the Riemannian volume form, $h_M$ denote the Cheeger constant, and $\lambda$ denote the bottom of the spectrum of the Laplacian of $M$.
\begin{enumerate}
	\item Let $g:M\to \R$ be an arbitrary smooth map with $\int_M g^2\ dV=1$ and $\int_M g\ dV=0$. The
	goal is to show that $\int_M |\nabla g|^2\ dV\geq h^2_M/4$
	\item Translate $g$ to get a map $f$ such that the volumes of $M_+:=\set{f\leq 0}$ and $M_-:=\set{f\geq 0}$ are the same.
	Note that $\int_M|\nabla g|^2\ dV\geq \int_M|\nabla f|^2\ dV$, so it suffices to show that the latter dominates $h^2_M/4$.
	\item Use the Cauchy-Schwarz inequality to get
	\begin{equation*}
	\int_M |\nabla f|^2\ dV\geq \frac{1}{4} \lrb{\int_{M^+}|\nabla f^2|\ dV + \int_{M_+} |\nabla f^2|\ dV}^2
	\end{equation*}
	\item Use the Co-area formula to write $\int_{M_\pm}|\nabla f^2|\ dV$ as integrals of the areas of fibers (slices of $M$) of $f^2$, and then use the definition of the Cheeger constant to finish.
\end{enumerate}
We adapt the above proof  to the case of graph limits by developing a suitable co-area formula with `volume' of a subset $A$ of $I$ meaning the `sum of degrees of the vertices in $A$' and the `area' of a slice being the `number of edges crossing the slice.' In the process we outline the more {\bf geometric content} of graph limits by explicitly describing the differential operators $d$ (an analog of the gradient operator) and the adjoint $d^\ast$ (an analog of the divergence operator).  An implicit and partial aim of this paper is to make at least some parts of the beautiful theory of graph limits accessible to a geometrically inclined audience. We have therefore provided some of the arguments in complete (possibly painful!) detail.

The proof we give for the Cheeger inequality also has some philosophical similarity with  the combinatorial proof of \cite[Thereom 2.2]{fan_chung_spectral_graph_theory}. However the techniques do not  quite apply here.
The proof there starts with 
 an eigenfunction of the Laplacian and uses a reordering of vertices, neither of which can be done  in our situation. These technical difficulties are circumvented by using the co-area formula.

\section{Preliminaries}\label{prel} In this section, we summarize general facts about graphons and graphings that we shall need in the paper. Most of the material is from the book by Lovasz \cite{lovasz_large_networks}, but in the subsection on graphings below we deduce a few elementary consequences as well as a slightly different perspective from that in \cite{lovasz_large_networks}.
\subsection{Preliminaries on graphons}\label{prel-graphon}
We summarize the relevant material from  \cite[Chapter 7]{lovasz_large_networks}.
Let $I$ denote the unit interval $[0, 1]$ and $\mu_L$ denote the Lebesgue measure on $I$.
A function $W:I^2\to I$ is said to be a \define{graphon} if $W$ is measurable and symmetric, that is, $W(x, y)=W(y, x)$ for all $x, y\in I$.
Given a graphon $W$, we define for each $x\in I$ the \define{degree} of $x$ as
\begin{equation}
d_W(x) = \int_0^1 W(x, y)\ dy
\end{equation}
A graphon $W$ is said to be \define{regular} if $d_W$ is constant $\mu_L$-a.e.
For two measurable subsets $A$ and $B$ of $I$, we define
\begin{equation}
e_W(A, B) = \int_{A\times B} W
\end{equation}
Thus, $e_W(A, B)$ is the total weight of edges between $A$ and $B$.
For a measurable subset $A$ of $I$, the \define{volume} of $W$ over $A$ is defined as
\begin{equation}
\vol_W(A) = \int_{A\times I} W = e_W(A, I)
\end{equation}
Thus, $\vol_W(A)$ measures the total weight of edges emanating from $A$.

A graphon is said to be \define{connected} if for all measurable subsets $A$ of $I$ with $0< \mu_L(A)< 1$ we have $e_W(A, A^c) \neq 0$.
Note that if $W$ is connected then $d_W>0$ a.e.

\subsection{Preliminaries on graphings}\label{prel-graphing}

Let $I$ denote the unit interval $[0, 1]$.
\begin{defn}\cite[Chapter 18]{lovasz_large_networks}
A \define{bounded degree Borel graph} on $I$ is a pair $(I, E)$, where $E$ is a symmetric measurable subset of $I^2$ such that there is a positive integer $D$ satisfying
\begin{equation}
|\set{y\in I:\ (x, y)\in E}|\leq D
\end{equation}
for all $x\in I$.
\end{defn}
In other words, the number of neighbors of each point in $I$ is at most $D$.
Given a bounded degree Borel graph $(I, E)$, we have a \define{degree function} $\deg:I\to \R$ defined as
\begin{equation}
\deg(x) = |\set{y\in I:\ (x, y)\in E}|
\end{equation}
For any measurable subset $A$ of $I$ we define $\deg_A:I\to \R$ as
\begin{equation}
\deg_A(x) = |\set{y\in A:\ (x, y)\in E}|
\end{equation}
It is proved in \cite[Lemma 18.4]{lovasz_large_networks} that the map $\deg_A$ is a measurable function for any measurable set $A\subseteq I$.
Note that $\deg$ is nothing but $\deg_I$.

\begin{defn}\cite[Chapter 18]{lovasz_large_networks}
A \define{graphing} is a triple $G=(I, \mu, E)$ such that $(I, E)$ is a bounded degree Borel graph, and $\mu$ is a probability measure on $I$ such that
\begin{equation}
\label{fubini condition of a graphing}
\int_A\deg_B(x)\ d\mu(x) = \int_B\deg_A(x)\ d\mu(x)
\end{equation}
for all measurable subsets $A$ and $B$ of $I$.
\label{def-graphing}
\end{defn}
Given a graphing $G=(I, \mu, E)$, the measure $\mu$ allows us to define a measure $\eta$ on $I^2$ as follows.
For each measurable rectangle $A\times B\subseteq I^2$, we define $$\eta(A\times B)=\int_A\deg_B(x) \ d\mu(x).$$
Equation \ref{fubini condition of a graphing} ensures that $\eta(A\times B)=\eta(B\times A)$.
By Caratheodory extension, we get a measure $\eta$ on the Borel $\sigma$-algebra of $I^2$.
As proved in \cite[Lemma 18.14]{lovasz_large_networks}, the measure $\eta$ is concentrated on $E$.

A fundamental result proved in \cite[Theorem 18.21]{lovasz_large_networks} is that  every graphing can be decomposed as a disjoint union of finitely many graphings, each having degree $\deg(x)$ bounded by 1 for all $x$.
More precisely,

\begin{theorem}
	\label{strand theorem}
	Let $G=(I, \mu, E)$ be a graphing.
	Then there exist measurable subsets $A_1 , \ldots, A_k\subseteq I$ and $\mu$-measure preserving involutions $\vp_i:A_i\to A_i$ such that
	\begin{equation}
	E=\bigsqcup_{i=1}^k \set{(x, \vp_i(x)):\ x\in A_i}
	\end{equation}
\end{theorem}

We can pictorially represent a graphing $G=(I, \mu, E)$ by drawing the edge set $E$ of $G$ in the unit square.
Each subset $\set{(x, \vp_i(x)):\ x\in A_i}$ can be thought of as a ``strand" in $I^2$.
Thus the previous theorem allows us to think of a graphing as a disjoint union of strands in the unit square.
When the degree bound of a graphing is $1$, we may say that the graphing consists of a \emph{single strand}.

The measure $\eta$ counts the number of edges in any measurable subset of $S\subseteq I^2$.
When $S=A\times B$ is a rectangle, we count the number of strands in $S$ each vertical line cuts, and integrate this count against $d\mu$.
This is immediate from the definition of $\eta$.
This extends to arbitrary $S$, as the following lemma shows.

\begin{lemma}
	\label{what eta really is}
	Let $S$ be any measurable subset of $I^2$.
	Then
	\begin{equation}
	\label{expression of eta in terms of mu}
	\eta(S) = \int_I\sum_y\chi_{E\cap S}(x, y) \ d\mu(x)
	\end{equation}
\end{lemma}
\begin{proof}
	First let us see why the integral on the RHS makes sense.
	Using Theorem \ref{strand theorem} we know that there exist $\mu$-measure preserving involutions $\vp_i:A_i\to A_i$, $i=1 , \ldots, k$, for some measurable subsets $A_i$ of $I$, such that
	\begin{equation}
	E=\bigsqcup_{i=1}^k \set{(x, \vp_i(x)):\ x\in A_i}
	\end{equation}

	Hence,
	\begin{equation}
	\sum_{y\in I} \chi_{E\cap S}(x, y) = \sum_{i=1}^k \chi_{S}(x, \vp_i(x))
	\end{equation}
	Thus the integrand in the RHS of Equation \ref{expression of eta in terms of mu} is a sum of finitely many non-negative measurable functions $I\to \R$ and therefore the RHS of Equation \ref{expression of eta in terms of mu} is well-defined.

	Let $\nu(S)$ be the RHS of Equation \ref{what eta really is}.
	Let us verify that $\nu$ is a measure on the Borel $\sigma$-algebra of $I^2$.
	So let $S=\bigsqcup_{j=1}^\infty S_j$ be a countable disjoint union of measurable sets.
	Then
	\begin{align}
	\begin{split}
	\chi_S(x, \vp_i(x)) &= \sum_{j=1}^\infty \chi_{S_j}(x, \vp_i(x))\\
	\ra \sum_{i=1}^k \chi_S(x, \vp_i(x)) &= \sum_{i=1}^k \sum_{j=1}^\infty \chi_{S_j}(x, \vp_i(x)) =  \sum_{j=1}^\infty \sum_{i=1}^k\chi_{S_j}(x, \vp_i(x))\\
	\ra \int_I\sum_{i=1}^k \chi_S(x, \vp_i(x)) &= \int_I \sum_{j=1}^\infty \sum_{i=1}^k\chi_{S_j}(x, \vp_i(x))\ d\mu(x)\\
	\end{split}
	\end{align}
	This implies that
	\begin{align}
	\begin{split}
	\nu(S)&= \int_I\lim_{n\to \infty} \sum_{j=1}^n \lrp{\sum_{i=1}^k\chi_{S_j}(x, \vp_i(x))} d\mu(x)\\
	&= \lim_{n\to \infty}\int_I\sum_{j=1}^n\lrp{ \sum_{i=1}^k\chi_{S_j}(x, \vp_i(x))} d\mu(x)\\
	&= \lim_{n\to \infty}\lrb{\sum_{j=1}^n\int_I\lrp{ \sum_{i=1}^k\chi_{S_j}(x, \vp_i(x))} d\mu(x)}\\
	\ra \nu(S) &= \lim_{n\to \infty} \sum_{j=1}^n \nu(S_j) = \sum_{j=1}^\infty \nu(S_j)\\
	\end{split}
	\end{align}
	showing that $\nu$ is countably additive and is therefore a measure.
	Now let $S=A\times B$ be a measurable rectangle.
	Then
	\begin{align}
	\begin{split}
	\nu(S) &= \int_I \sum_{y} \chi_{E\cap(A\times B)}(x, y) \ d\mu(x)\\
	&= \int_A \sum_y\chi_{E\cap (I\times B)}(x, y) \ d\mu(x)\\
	&= \int_A \deg_B(x)\ d\mu(x)\\
	&= \eta(A\times B)
	\end{split}
	\end{align}
	So $\nu$ agrees with $\eta$ on the measurable rectangles.
	But since extension of a finitely additive and countably sub-additive measure on the algebra of measurable rectangles to the Borel $\sigma$-algebra is unique (Caratheodory Extension Theorem), we must have that $\nu=\eta$ and we are done.
\end{proof}

If we have a non-negative map $\psi:I^2\to \R$, then by definition of integration we have that
\begin{align}
\begin{split}
\int_{I^2}\psi\ d\eta &= \lim_{n\to \infty}\int_{I^2} \psi_n\ d\eta\\
\end{split}
\end{align}
where $(\psi_n)$ is a sequence of non-negative simple functions such that $\psi_n\uparrow \psi$.
This definition gives a theory of integration.
We could define another theory of integration by declaring the integral of $\psi$ to be equal to
\begin{equation}
\int_I \sum_{y\in I}\psi(x, y) \chi_E(x, y)\ d\mu(x)
\end{equation}
By Lemma \ref{what eta really is} these two theories agree on simple functions, and therefore are the same theories of integration.
So for any $\psi\in L^1(I^2, \eta)$ we have
\begin{equation}
\label{what is eta in terms of mu}
\int_{I^2} \psi(x, y) \ d\eta(x, y) = \int_I \sum_y\psi(x, y)\chi_E(x, y)\ d\mu(x)
\end{equation}

\section{Cheeger Constant, Laplacian, and the Bottom of the Spectrum for a Graphon}\label{graphon}

\subsection{Cheeger Constant for Graphons}
\begin{defn}\label{def:cheeger-graphon}
Given a graphon $W$, we define the \define{Cheeger constant}  of $W$ as
\begin{equation}
	h_W= \inf_{A\subseteq I:\ 0< \mu_L(A)< 1}\frac{e_W(A, A^c)}{\min\set{\vol_W(A), \vol_W(A^c)}}
\end{equation}
It will be convenient to denote the quantity
\begin{equation}
	\frac{e_W(A,A^c)}{\min\set{\vol_W(A), \vol_W(A^c)}}
\end{equation}
as $h_W(A)$.
A  symmetrized version of the above constant, which we call the \define{symmetric Cheeger constant} is defined as
\begin{equation}
	g_W = \inf_{A\subseteq I:\ 0< \mu_L(A)< 1}\frac{e_W(A, A^c)}{\vol_W(A)\vol_W(A^c)}
\end{equation}
\end{defn}
The analogue of $g_W$ for finite graphs is called the \emph{averaged minimal cut} in \cite{friedland_nabben_weighted_cheeger}.
Note that the above defined constants exist for connected graphons.

\begin{lemma}
	\label{lemma:bound on the cheeger constant of a graphon}
	Let $W$ be a connected graphon.
	Then $h_W\leq 1/2$.
\end{lemma}
\begin{proof}
	Let $A$ be the interval $[0, 1/2]$ and write $\eta_L$ to denote the Lebesgue measure on $I^2$.
	Define $S:I\to I$ as $S(x)=2x\pmod{1}$, and write $A_n$ to denote $S^{-n}(A)$.
	Thus $S$ is strong mixing, and hence so is $T:=S\times S:I^2\to I^2$.
	Fix $\varepsilon>0$.
	The strong mixing property of $T$ gives that
	\begin{equation}
		\lim_{n\to \infty} \int_{I^2} (\chi_{A\times A^c}\circ T^n)W\ d\eta_L = \lrp{\int_{I^2} \chi_{A\times A^c}\ d\eta_L} \lrp{\int_{I^2} W\ d\eta_L} = \vol_W(I)/4
	\end{equation}
	and
	\begin{equation}
		\lim_{n\to \infty} \int_{I^2} (\chi_{A\times I}\circ T^n)W\ d\eta_L = \lrp{\int_{I^2} \chi_{A\times A^c}\ d\eta_L} \lrp{\int_{I^2} W\ d\eta_L} = \vol_W(I)/2
	\end{equation}
	Therefore for $n$ large enough we have
	\begin{equation}
		\lrmod{\int_{A_n\times A_n^c} W\ d\eta_L - \vol_W(I)/4}< \varepsilon
	\end{equation}
	and
	\begin{equation}
		\lrmod{\int_{A_n\times I} W\ d\eta_L - \vol_W(I)/2} < \varepsilon 
	\end{equation}
	From the last equation we also get
	\begin{equation}
		\lrmod{\int_{A_n^c\times I} W\ d\eta_L - \vol_W(I)/2} < \varepsilon 
	\end{equation}
	So we see that the ratio
	\begin{equation}
		\frac{\int_{A_n\times A_n^c} W\ d\eta_L
		}
		{
			\min\set{\int_{A_n\times I} W\ d\eta_L, \int_{A_n^c\times I} W\ d\eta_L}
		}
		= h_{W}(A_n)
	\end{equation}
	can be made arbitrarily close to $ \frac{1}{4}\vol_W(I)/ \frac{1}{2} \vol_W(I) = 1/2$ for $n$ suitably large.
	But $h_W\leq h_W({A_n})$ and so we conclude that $h_W\leq 1/2$.
\end{proof}

There certainly exist graphons with Cheeger constant $1/2$, for example the graphon which takes the value $1$ everywhere.

\subsection{Definition of d, d*, and Laplacian of a Graphon.}
Let $W$ be a connected graphon.
Define
\begin{equation}
	E=\set{(x, y)\in I^2:\ y>x}, \qquad E_W=\set{(x, y)\in E:\ W(x, y) >0}
\end{equation}
The set $E$ can be thought of as an orientation of all the ``edges".
The set $E_W$ disregards the oriented edges which have zero weight.

Define a measure $\nu$ on $I$ as
\begin{equation}
	\nu(A) = \int_A d_W(x)\ dx = \vol_W(A)
\end{equation}
for all  measurable subsets $A$ of $I$.
In other words, the Radon-Nikodym derivative of $\nu$ with respect to the Lebesgue measure is $d_W$.
Clearly, $\nu$ is absolutely continuous with respect to the Lebesgue measure on $I$.
The connectedness of $W$ implies that the Lebesgue measure is also absolutely continuous with respect to $\nu$.
Thus we may talk about null sets in $I$ unambiguously.
This also says that $L^\infty(I, \nu)=L^\infty(I, \mu_L)$, and thus we write these simply as $L^\infty(I)$.

Similarly, define a measure $\eta$ on $E_W$ as
\begin{equation}
	\eta(S) = \int_S W(x, y) \ dxdy
\end{equation}
for all measurable subsets $S$ of $E_W$.
So the Radon-Nikodym derivative of $\eta$ with respect to the Lebesgue measure is $W$.
These measures give rise to Hilbert spaces $L^2(I, \nu)$ and $L^2(E_W, \eta)$, the inner products on which will be denoted by $\ab{\cdot, \cdot}_v$ and $\ab{\cdot, \cdot}_e$ respectively.
Explicitly
\begin{equation}
	\ab{f, g}_v = \int_0^1 f(x)g(x) d_W(x)\ dx
\end{equation}
for all $f, g\in L^2(I, \nu)$, and
\begin{equation}
	\ab{\vp, \psi}_e = \int_{E_W}\vp\psi W = \int_0^1\int_x^1 \vp(x, y) \psi(x, y) W(x, y)\ dy dx
\end{equation}
for all $\vp, \psi\in L^2(E_W)$.
The standard inner products on $L^2(I, \mu_L)$ and $L^2(I^2, \mu_L\otimes \mu_L)$ will be denoted by $\ab{\cdot, \cdot}_{L^2(I)}$ and $\ab{\cdot, \cdot}_{L^2(I^2)}$.

%


Define a map $d:L^2(I,\nu)\to L^2(E_W, \eta)$ as
\begin{equation}
	(df)(x, y) = f(y) - f(x)
\end{equation}
for all $f\in L^2(I, \nu)$.
The map $d$ can be thought of as a gradient which measures the change in $f$ as we travel from the tail of an edge to the head.
We need to check that $df$ actually lands in $L^2(E_W, \eta)$ for any given member of $L^2(I, \nu)$.
This and more is proved in the following lemma.

\begin{lemma}\label{d-cont}
	The map $d$ is continuous.
\end{lemma}
\begin{proof}
	We want to show that $d$ is bounded.
	Let $f\in L^2(I, \nu)$.
	Then
	\begin{align}
		\begin{split}
			 \norm{df}_e^2 &= \ds\int_{E_W} (df)^2 W\\
			 &= \int_0^1\int_x^1 (f(y)-f(x))^2W(x, y)\ dydx\\
			 &\leq \int_0^1\int_0^1 (f(y)-f(x))^2W(x, y)\ dydx\\
			 &\leq \int_0^1\int_0^1 f^2(y)W(x, y)\ dy dx + \int_0^1\int_0^1 f^2(x) W(x, y)\ dy dx \\&\hskip15em\relax+ 2\int_0^1\int_0^1 |f(x)f(y)|W(x, y)\ dydx\\
			 &= 2\int_0^1f^2(x)d_W(x)\ dx + 2\int_0^1\int_0^1 |f(x)f(y)|W(x, y)\ dydx
	 	\end{split}
	\end{align}
	The first term is the same as $2\norm{f}_v^2$.
	So we need to bound the second term.
	Let $\alpha, \beta:I^2\to\R$ be defined as
	\begin{equation}
		\alpha(x, y)= |f(x)|\sqrt{W(x, y)}, \qquad \beta(x, y) = |f(y)| \sqrt{W(x, y)}
	\end{equation}
	The fact that $f\in L^2(I, \nu)$ implies that $\alpha, \beta\in L^2(I^2)$.
	Then we have
	\begin{equation}
		\ds\int_0^1\int_0^1 |f(x)f(y)|W(x, y)\ dydx = \ab{\alpha, \beta}_{L^2(I^2)}
	\end{equation}
	But now by Cauchy-Schwarz inequality we have
	\begin{align}
		\begin{split}
			\ab{\alpha, \beta}_{L^2(I^2)} &\leq \norm{\alpha}_{L^2(I^2)}\norm{\beta}_{L^2(I^2)}\\
			&= \ds\lrp{\int_0^1\int_0^1 f^2(x)W(x, y) \ dydx }^{1/2}\lrp{\int_0^1\int_0^1 f^2(y)W(x, y)\ dx dy}^{1/2}\\
			&= \norm{f}_v^2
		\end{split}
	\end{align}
	We conclude that $\norm{df}_e\leq 2\norm{f}_v$.
	This shows that $d$ is continuous.
\end{proof}

The above lemma shows that $d^*$, the adjoint\footnote{Here, again, the adjoint is taken with respect to the Hilbert space structure coming from  $\ab{\cdot, \cdot}_v$ and $\ab{\cdot, \cdot}_e$.} of $d$, exists.
We now calculate it explicitly.
Let $f\in L^2(I, \nu)$ and $\vp\in L^2(E_W, \eta)$ be arbitrary.
We have

%
\begin{align}
	\begin{split}
		\ab{df, \vp}_e &= \int_0^1\int_x^1 df(x, y) \vp(x, y)W(x, y)\ dydx\\
		&= \int_0^1\int_x^1 (f(y)-f(x)) \vp(x, y) W(x, y)\ dy dx\\
		&= \int_0^1\int_x^1 f(y) \vp(x, y) W(x, y) \ dydx - \int_0^1\int_x^1 f(x) \vp(x, y) W(x, y)\ dy dx\\
		&\reason{Fubini}{=}  \int_0^1\int_0^y f(y) \vp(x, y) W(x, y)\ dx dy - \int_0^1\int_x^1 f(x) \vp(x, y) W(x, y)\ dy dx\\
		&= \int_0^1\int_0^x f(x) \vp(y, x) W(x, y) \ dy dx - \int_0^1\int_x^1 f(x) \vp(x, y) W(x, y) \ dy dx\\
		&= \int_0^1 f(x)\lrb{ \int_0^x \vp(y, x) W(x, y)\ dy - \int_x^1\vp(x, y) W(x, y)\ dy}dx
	\end{split}
\end{align}
On the other hand we have
\begin{equation}
	\ab{f, d^*\vp}_v = \int_0^1 f(x) d^*\vp(x) d_W(x) \ dx
\end{equation}
Thus we have
\begin{equation}
	(d^*\vp)(x) =  \frac{1}{d_W(x)}\lrb{\int_0^x \vp(y, x) W(x, y)\ dy - \int_x^1\vp(x, y) W(x, y)\ dy}
\end{equation}
wherever $d_W(x) \neq 0$.
We set $(d^*\vp)(x)=0$ if $d_W(x)=0$.

\begin{rmk}
We have adapted the language of differential forms above so that we think of the map $$d:C^0(W) \to C^1(W)$$ as an exterior derivative from $0-$forms (i.e. functions on the vertex set) to $1-$forms (i.e. functions on the set of directed edges). Then $d^\ast$ is the adjoint map using the Hodge $\ast$: $$d^\ast:C^1(W) \to C^0(W).$$

 Alternately, in the presence of  inner products on both the vertex and edge-spaces (the situation  here) we may think of $d$ as an analog of  the gradient  operator ($\rm grad$ or $\nabla$) in classical vector calculus and 
$d^\ast$ as an analog of the divergence operator $\rm div$.
\end{rmk}

Define the \define{Laplacian} of $W$ as $\lap_W=d^*d$.
We may drop the subscript when there is no confusion.
For $f\in L^2(I, \nu)$, we calculate $(\lap_Wf)(x)$.
\begin{align}
	\begin{split}
		(\lap_Wf)(x) &= (d^*df)(x)\\
	   	&= \frac{1}{d_W(x)} \lrb{\int_0^x df(y, x) W(x, y) \ dy-\int_x^1 df(x, y)W(x, y)\ dy} \\
		&=  \frac{1}{d_W(x)}\lrb{\int_0^1 (f(x)-f(y)) W(x, y) \ dy}\\
		&=  f(x) - \frac{1}{d_W(x)} T_Wf(x)
	\end{split}
\end{align}
where $T_W:L^2(I, \nu)\to L^2(I, \nu)$ is a linear map defined as
\begin{equation}
	\label{definition of adjacency operator}
	(T_Wf)(x) = \ds\int_0^1 W(x, y)f(y)\ dy
\end{equation}
The map $T_W$ is well-defined.
Indeed, the integral on the RHS of Equation \ref{definition of adjacency operator} exists.
To see this, let $\constone$ denote the constant map $I\to \R$ which takes all points to $1$.
Then $\constone\in L^2(I, \nu)$, and thus
\begin{align}
	\begin{split}
		\ab{|f|, \constone}_v = \int_0^1 |f(y)| d_W(y)\ dy< \infty\\
		\ra \int_0^1 |f(y)|\lrb{\int_0^1 W(x, y)\ dx} dy < \infty\\
		\ra \int_0^1\lrb{\int_0^1 |f(y)|W(x, y)\ dy}dx < \infty
	\end{split}
\end{align}
Therefore $\int_0^1|f(y)| W(x, y)\ dy$ is almost everywhere finite and consequently $(T_Wf)(x)$ exists.
It is also easy to check (using Cauchy-Schwarz) that $T_Wf$ lies in $L^2(I, \nu)$.
Therefore we have $\lap_W = I- \frac{1}{d_W} T_W$.

\subsection{Bottom of the Spectrum of a Graphon}
Let us see what is the multiplicity of the singular value $0$ of the Laplacian of a connected graphon $W$.
For $f\in L^2(I, \nu)$, we have $\lap f=0$ if and only if $df=0$.
We claim that $df=0$ if and only if $f$ is constant (up to a set of measure zero).
Clearly, if $f$ is constant, then $df=0$.
Conversely, assume that $df=0$.
Thus
\begin{equation}
	\int_{E_W} (df)^2 W = \int_E (df)^2 W = \int_0^1\int_x^1(f(y)-f(x))^2W(x, y)\ dy dx =0
\end{equation}
which implies that
\begin{equation}
	\int_0^1\int_0^1 (f(y)-f(x))^2W(x, y)\ dy dx=0
\end{equation}
For each $t\in \R$, let $S_t=f^{-1}(t, \infty)$.
From the last equation we have
\begin{equation}
	\int_{S_t^c\times S_t} (f(y)-f(x))^2 W(x,y) \ dy dx=0
\end{equation}
which implies that $(f(y)-f(x))^2W(x, y)$ is a.e. $0$ on $S_t^c\times S_t$.
But $f(y)-f(x)\neq 0$ for all $(x, y)\in S_t^c\times S_t$, which means that $W=0$ a.e. on $S_t^c\times S_t$.
The connectedness of $W$ then implies that either $S_t$ or $S_t^c$ has measure $0$.
So our claim follows from the following lemma.

\begin{lemma}\label{ess-const-graphon}
	Let $f:I\to \R$ be a measurable function such that for all $t\in \R$ we have either $f^{-1}(-\infty, t]$ or $f^{-1}(t, \infty)$ has measure $0$.
	Then $f$ is essentially constant.
\end{lemma}
\begin{proof}
	Let
	\begin{equation}
		t_0=\inf\set{t\in \R:\ f^{-1}(-\infty, t] \text{ is full measure}}
	\end{equation}
	Then $t_0\neq -\infty$.
	This is because $I=\bigsqcup_{n\in \Z} f^{-1}(n, n+1]$.
	Thus $f^{-1}(n, n+1]$ has positive measure for some integer $n$, and this $n$ cannot exceed $t_0$.
	Also, by definition of $t_0$, we have that $f^{-1}(-\infty, t_0-1/n]$ has measure $0$ for each $n$.
	Thus $f^{-1}(-\infty, t_0)$ also has measure $0$.
	Again, by definition of $t_0$ we have that $f^{-1}(t_0+1/n, \infty)$ has measure zero for all $n$, and thus $f^{-1}(t_0, \infty)$ has measure zero.
	So we conclude that $f^{-1}(t_0)$ has full measure.
\end{proof}

So we have shown that the eigenfunctions corresponding to $0$ are precisely the essentially constant functions.
In other words, the eigenspace of $\lap$ corresponding to $0$ is generated by $\constone$, the constant function taking value $1$ everywhere.
The {\bf bottom of the spectrum} denoted $\lambda_W$ is therefore given by the following Rayleigh quotient:
\begin{equation}
	\lambda_W = \inf_{f\in \constone^\perp_v, f\neq 0} \frac{\ab{f, \lap f}_v}{\ab{f, f}_v} = \inf_{f\in \constone^\perp_v, f\neq 0} \frac{\norm{df}^2_e}{\norm{f}^2_v}
\end{equation}
(Here, $\constone^\perp_v$ denotes the orthogonal complement of $\constone$ with respect to the inner product $\ab{\cdot, \cdot}_v$).

\section{Finite graphs and graphons} The purpose of this section is to explore the relationship between the Cheeger constant of a  finite graphs with that of a canonically associated graphon. Similarly we study the relationship between the bottom of the spectrum of a  finite graphs with that of its canonically associated graphon.
\subsection{Cheeger constant of a graph versus that of the corresponding graphon}
\label{section:cheeger constant of a graph versus a graphon}
In what follows, by a \define{weighted graph} we mean a pair $(V, w)$, where $w:V\times V\to [0, 1]$ is a symmetric map. 
Every weighted graph $G$ naturally gives rise to a graphon.
It is  natural  to ask about the relation between their Cheeger constants.
Clearly that $h_W\leq h_G$.
The aim of this section is to put a lower bound on the ratio $h_W/h_G$ when $G$ is \emph{loopless}, where a \define{loopless weighted graph} is a weighted graph $(V, w)$ such that $w(x, x)=0$ for all $x\in V$.
We will also assume that all the weighted graphs considered are connected.
This means that whenever we partition the vertex set into two parts, the total weight of the cut is positive.
The volume of a vertex $v$ of a weighted graph $(V, w)$ is defined as $\vol(v)=\sum_{u\in V}w_{uv}$.
We also define $\vol(G)=\sum_{v\in V}\vol(v)$.

Given any set $V$, a \define{fractional partition} of $V$ is a pair $(\rho, \eta)$, where $\rho, \eta:V\to I$ are functions such that $\rho(u)+\eta(u)=1$ for all $u\in V$.
Note that a true partition of $V$ (into two parts) can be thought of as a fractional partition $(\rho, \eta)$ such that $\rho$ and $\eta$ takes values in $\set{0, 1}$.

Let $G=(V=[n], w)$ be a weighted graph.
We define the \emph{fractional Cheeger constant} of $G$ as follows:
For a fractional partition $(\rho, \eta)$ of $V$, we define
\begin{equation}
	\label{equation: definition of fractional cheeger constant}
	\tilde h(G;\ \rho, \eta) = \frac{\sum_{u, v\in V} \rho(u)\eta(v)w(u, v)}{\min\set{\sum_{u\in V} \rho(u)\vol(u), \sum_{v\in V} \eta(v)\vol(v)}}
\end{equation}
Of course, the above is well-defined only when $\norm{\rho} := \sum_{u\in V} \rho(u)\vol(u)\neq 0$ and $\norm{\eta}:=\sum_{v\in V} \eta(v)\vol(v)\neq 0$, and throughout we will tacitly assume this condition.
The \define{fractional Cheeger constant} of $G$ is defined as
\begin{equation}
	\tilde h_G = \inf_{(\rho, \eta)} \tilde h(G;\ \rho, \eta)
\end{equation}
where the infimum runs over all fractional partitions $(\rho, \eta)$ of $V$.
Note that the Cheeger constant of the graphon corresponding to a graph $G$ is the same as the fractional Cheeger constant of the graph $G$.
The use of the notion of fractional Cheeger constant is just for convenience.
{Therefore, by Lemma \ref{lemma:bound on the cheeger constant of a graphon} the fractional Cheeger constant of any weighted graph is at most $1/2$}.
\footnote{This can be seen directly.
One can achieve the value $1/2$ by choosing a fractional partition which puts half of each vertex on one side and the other half on the other side.}

\subsubsection{Realization of Fractional Cheeger}

\begin{lemma}
	Let $G=(V=[n], w)$ be a weighted graph.
	Then the fractional Cheeger constant of $G$ is realized by a  fractional partition.
\end{lemma}
\begin{proof}
	Let $\tilde h$ be the fractional Cheeger constant of $G$ and $(\rho_1, \eta_1), (\rho_2, \eta_2), (\rho_3, \eta_3), \ldots$ be a sequence of fractional partitions of such that
	\begin{equation}
		\label{equation:simple equation}
		\tilde h(G; \rho_k, \eta_k) \leq \tilde h + 1/k\leq 1/2 + 1/k
	\end{equation}
	for all $k$.
	Without loss of generality, assume that $\norm{\rho_k}\leq \vol(G)/2$ for all $k$.
	Since each $\rho_n$ can be thought of as a member of the compact metric space $I^n$, we may assume, by passing to a subsequence if necessary, that $\rho_n\to \rho\in I^n$.
	If $\norm{\rho}>0$ then it is clear that $\tilde h(G; \rho, 1-\rho)= \tilde h$.
	So we may assume that $\rho(i)=0$ for all $i$.
	Then for all large enough $k$ we have $\rho_k(i)< 1/3$.
	Therefore
	\begin{align}
		\begin{split}
			\sum_{i, j=1}^n \rho_k(i)(1-\rho_k(j)) w_{ij} &=\sum_{i=1}^n\rho_k(i)\lrp{\sum_{j=1}^n(1-\rho_k(j))w_{ij}}\\
			&\geq \sum_{i=1}^n \rho_k(i)\lrp{\sum_{j=1}^n 2 w_{ij}/3} = 2/3 \sum_{i=1}^n \rho_k(i)\vol(i)
		\end{split}
	\end{align}
	Therefore
	\begin{equation}
		\tilde h(G; \rho_k, \eta_k)
		=\frac{\sum_{i, j=1}^n \rho_k(i)(1-\rho_k(j)) w_{ij}}{\sum_{i=1}^n\rho_k(i)\vol(i)} \geq \frac{2}{3}
	\end{equation}
	Thus Equation \ref{equation:simple equation} gives $1/2+1/k\geq 2/3$ for all large enough $k$.
	This is a contradiction.
\end{proof}


Next,
define functions $f:I^n\to \R$ and $s:I^n\to \R$ as follows:
\begin{equation}
	f(x_1 , \ldots, x_n) = \sum_{i, j=1}^n x_i(1-x_j)w_{ij}
\end{equation}
and
\begin{equation}
	s(x_1 , \ldots, x_n) = x_1\vol(1) + \cdots + x_n\vol(n)
\end{equation}
Taking the partial derivative of $f$ with respect to $x_p$, we have
\begin{equation}
	\partial f/\partial x_p = \sum_{j=1}^n (1-2x_i)w_{pj}
\end{equation}
and thus
\begin{equation}
	\partial^2f/\partial x_p^2 = 0
\end{equation}
for any $1\leq p\leq n$ since $w_{pp}=0$.

\begin{lemma}
	\label{lemma:the kth derivative}
	For $k\geq 1$, we have
	\begin{equation}
		\ds\frac{\ds\partial^k(f/s)}{\partial x_p^k} = \frac{(-1)^{k+1}k!\vol(p)^{k-1}}{s^{k+1}}\lrp{ s \frac{\partial f}{\partial x_p} - \vol(p)f}
	\end{equation}
\end{lemma}
\begin{proof}
	Induction.
\end{proof}

\begin{lemma}
	\label{lemma:realization of cheeger at equipartition}
	Let $G$ be a loopless weighted graph whose fractional Cheeger constant is strictly less than its Cheeger constant.
	Then the fractional Cheeger constant of $G$ can be achieved at a fractional partition $(\rho, \eta)$ such that $\norm{\rho}=\norm{\eta}$.
\end{lemma}
\begin{proof}
Suppose that the fractional Cheeger constant of $G$ is achieved at a fractional partition $(\rho, \eta)$ such that $\norm{\rho}< \norm{\eta}$ and write $a_i=\rho(i)$.
Without loss of generality, assume $a_1 \leq \cdots \leq a_n$.
Some $a_i$ must be strictly between $0$ and $1$, for otherwise the fractional Cheeger constant of $G$ would be equal to the Cheeger constant of $G$.
Say $p\in [n]$ is such that $0<a_p<1$.
Now
\begin{equation}
	\frac{\partial(f/s)}{\partial x_p}
	= \frac{1}{s^2}\lrp{s \frac{\partial f}{\partial x_p} -\vol(p) f}
\end{equation}
If this quantity were not zero, then we could perturb $a_p$ slightly to decrease the value of $f/s$, which would mean that the fractional Cheeger constant of $G$ could be reduced, contrary to the choice of $(\rho , \eta)$.
But this would contradict the fact that the fractional Cheeger constant is realized at $(x_1 , \ldots, x_n)=(a_1 , \ldots, a_n)$.
But then by Lemma \ref{lemma:the kth derivative}, we see that all the partial derivatives of $f/s$ with respect to $x_p$ vanish at the point $(a_1 , \ldots, a_p)$.
Since the function $f/s$ is analytic, this means that the function $f/s$ does not change when we perturb the $p$-th coordinate.
So we may increase it as much as we can, that is, we may push it all the way up to $1$ if $s$ does not cross $\vol(G)/2$ in the process, or stop as soon as $s$ hits the value $\vol(G)/2$.
If we hit $s=\vol(G)/2$ we stop since we have proved our claim.
Otherwise we can set $x_p=1$, and repeat the process for the remaining $q$'s for which $0<a_q<1$.
It cannot be the case that all $x_i$ will be either $0$ or $1$ at the end of this process, since if that were so then the fractional Cheeger constant of $G$ would be equal to the Cheeger constant of $G$, contrary to the hypothesis of the lemma, completing the proof.
\end{proof}

\subsubsection{Comparing the Cheeger constants}
\begin{lemma}
	\label{lemma:comparing cheeger constants}
	Let $G=(V=[n], w)$ be a loopless weighted graph.
	Then for all $1>\varepsilon>0$, we have that
	\begin{equation}
		\frac{\tilde h_G}{h_G} \geq \lrp{1- \frac{2\gamma}{\varepsilon ^2 n}}(1- \varepsilon)
	\end{equation}
	where
	\begin{equation}
		\gamma = \frac{\max\set{\vol(i):\ 1\leq i\leq n}}{\min\set{\vol(i):\ 1\leq i\leq n}}
	\end{equation}
\end{lemma}
%
%
\begin{proof}
	Let $h$ be the Cheeger constant of $G$ and $\delta h$ be the fractional Cheeger constant of $G$, where $0\leq  \delta\leq 1$.
	If $\delta=1$ then there is nothing to prove.
	So we assume that $\delta< 1$.
	Then from Lemma \ref{lemma:realization of cheeger at equipartition} we can find a fractional partition$(\rho, \eta)$ of $V$ which realizes the fractional Cheeger constant of $G$ and has the property that $\norm{\rho}=\norm{\eta}$.
	Write $p_i=\rho(i)$, so that $p_1\vol(1) + \cdots + p_n\vol(n)=\vol(G)/2$.
	Since the fractional Cheeger constant of $G$ is $\delta h$, we have $f(p_1 , \ldots, p_n)= \delta h \vol(G)/2$.

	Let $R_1 , \ldots, R_n$ be independent random variables such that $R_i$ takes the value $1$ with probability $p_i$ and takes the value $0$ with probability $1-p_i$.
	Write $L_j$ to denote $1-R_j$ for each $1\leq j\leq n$.
	\footnote{We can think of the tuple $(R_1 , \ldots, R_n)$ as a random true partition: If $R_i=1$ then the $i$-th vertex goes ``right" and if $L_i=1$ then the $i$-th vertex goes ``left."}
Let $Y=\sum_{i, j=1}^n R_i(1-R_j)w_{ij}$, and $Z=\sum_{i=1}^n R_i\vol(i)$.
It is clear that
\begin{equation}
	\E[Y] = f(p_1 , \ldots, p_n) = \delta h\vol(G)/2\quad  \text{ and }\quad \E[Z] = \vol(G)/2
\end{equation}
The variance of $Z$ is given by
\begin{align}
	\begin{split}
		\var(Z) &= \sum_{i=1}^n \E[R_i^2]\vol(i)^2 - \E[R_i]^2\vol(i)^2\\
		&= \sum_{i=1}^n p_i\vol(i)^2- p_i^2\vol(i)^2\\
		&\leq \vol_{\max}\sum_{i=1}^n p_i\vol(i)\\
		& = \vol_{\max}\vol(G)/2
	\end{split}
\end{align}
where $\vol_{\max}=\max\set{\vol(i):\ 1\leq i\leq n}$

Now let $\varepsilon$ be a positive number between $0$ and $1$.
By Chebychef's inequality we have
\begin{align}
	\begin{split}
		P(|Z-\vol(G)/2|\geq \varepsilon \vol(G)/2) &=P(|Z-\E[Z]|\geq \varepsilon \E[Z])\\
		&\leq \frac{\var(Z)}{\varepsilon^2 \E[Z]^2}\\
		&\leq \frac{\vol_{\max}\vol(G)/2}{\varepsilon^2 \vol(G)^2/4} = \frac{2\vol_{\max}}{\varepsilon^2\vol(G)}\leq \frac{\vol_{\max}}{\vol_{\min}} \frac{2}{\varepsilon^2 n}
	\end{split}
\end{align}
Write $\gamma$ to denote $\vol_{\max}/\vol_{\min}$.
So we have from the above equation that
\begin{equation}
		\label{repeat this equation}
	P(|Z-\vol(G)/2|\geq \varepsilon\vol(G)/2) \leq \frac{2\gamma}{\varepsilon^2 n}
\end{equation}
Thus with probability at least $1-2\gamma/(\varepsilon^2n)$ we have that $|Z- \vol(G)/2| \leq \varepsilon \vol(G)/2$.
But whenever $|Z- \vol(G)/2|\leq \varepsilon \vol(G)/2$, we have that
\begin{equation}
	\frac{Y}{(1-\varepsilon) \vol(G)/2}\geq \frac{Y}{\min\set{Z, \vol(G)-Z}} \geq h_G = h
\end{equation}
So with probability at least $1- 2\gamma/(\varepsilon^2 n)$ we have $Y\geq (1-\varepsilon) h\vol(G)/2$.
Therefore, since $Y$ takes only positive values, we have
\begin{equation}
	\lrp{1- \frac{2\gamma}{\varepsilon^2 n}} \frac{(1-\varepsilon) h\vol(G)}{2}
	\leq
	P\lrb{Y\geq \frac{(1-\varepsilon) h\vol(G)}{2}} \frac{(1-\varepsilon)h\vol(G)}{2}  \leq \E[Y] = \frac{\delta h\vol(G)}{2}
\end{equation}
This yields
\begin{equation}
	\label{equation:control on delta n}
	\lrp{1- \frac{2\gamma}{\varepsilon^2 n}} (1-\varepsilon) \leq \delta
\end{equation}
and we are done.
\end{proof}

\begin{rmk}
	The above result shows the elementary fact that there is no family $G_1, G_2, G_3 , \ldots$ of degree bounded graphs such that $h_{G_n}>1/2$ for all $n$, since $h_{W_{G_n}}\leq 1/2$ for all $n$.
\end{rmk}

\begin{rmk}
	The bound obtained in the above result is poor if $\gamma$ is of the order of $n$.
   However, if $G$ is a regular graph (more generally, a regular weighted loopless graph) with a large vertex set, then the above bound shows that the Cheeger constant of the graphon corresponding to $G$ is a good proxy for the Cheeger constant of $G$.
\end{rmk}

\begin{rmk}
	If $G$ is a regular weighted loopless graph, then using Azuma's inequality instead of Chebychef's, one gets an improved bound for $\delta$, namely
	\begin{equation}
		\lrp{1- \frac{2}{e^{n\varepsilon^2/8}}} (1-\varepsilon) \leq \delta, \quad \forall \varepsilon>0
	\end{equation}
\end{rmk}

\subsection{Bottom of the Spectrum of a graph versus that of the corresponding graphon}
\label{section:comparing the bottom of the spectrum}
Let $G=(V=[n], w)$ be a connected weighted graph and $W$ be the corresponding graphon.
We will show that $\lambda_W$ is same as the second eigenvalue of the normalized Laplacian of $G$.

Define the partition $\mc P$ of $I$ as
$$
	\mc P
	=
	\lrset{\left[0, \frac{1}{n}\right), \left[ \frac{1}{n}, \frac{2}{n}\right), \left[\frac{2}{n}, \frac{3}{n}\right), \ldots, \left[\frac{n-2}{n}, \frac{n-1}{n}\right), \left[\frac{n-1}{n}, 1\right]}
$$
and let $\mc A$ be the $\sigma$-algebra on $I$ generated by $\mc P$.
Also define inner product $\ab{-, -}_V$ on the vector space of all functions $V\to \R$ by declaring
\begin{equation}
	\ab{g, h}_V = \sum_{u\in V} g(u)h(u) \vol(u)
\end{equation}
Recall that the bottom of the spectrum of the Laplacian $\Delta_W$ of $W$ is defined as
\begin{equation}
	\lambda_W
	=
	\inf_{f\in \one^\perp_v: f\neq 0} \frac{\ab{\Delta_W f, f}_v}{\ab{f, f}_v}
	=
	\inf_{f\in \one^\perp_v: f\neq 0} \frac{\norm{df}_e^2}{\norm{f}_v^2}
\end{equation}
On the other hand, the smallest non-zero eigenvalue of the (normalized) Laplacian $\Delta_G$ of the graph $G$ is
\begin{equation}
	\lambda_G = \inf \frac{\ab{\Delta_G g, g}_V}{\ab{g, g}_V}
\end{equation}
where the infimum is taken over all nonzero $g:V\to \R$ such that $\ab{g, \one}_V=0$.
It is easy to see that if $f:I\to \R$ is any map such that $f|_P$ is constant for each $P\in \mc P$ and satisfies $\ab{f, \one}_v =0$, then
\begin{equation}
	\frac{\ab{\Delta_W f, f}_v}{\ab{f, f}_v} \geq \lambda_G
\end{equation}
and taking infimum over all such functions $f$ leads to an equality in the above.
So it is clear that $\lambda_G\geq \lambda_W$.
We show below that $\lambda_W\geq \lambda_G$, and hence $\lambda_G=\lambda_W$.
We will make use of the notion of conditional expectation, though that is not necessary.

Let $f\in \one^\perp_v$ be an arbitrary nonzero map in $L^2(I, \nu)$, where recall that $\nu$ is a measure on $I$ defined by setting $\nu(A)=\int_A d_W(x)\ dx$ for each Borel set $A$ in $I$.
Further assume that $\norm{f}_v^2=1$.
It is enough to show that $\norm{df}^2_e\geq \lambda_G$.
Let $F:I\to \R$ be the function defined as $F=\E[f|\mc A]$.
Now we have
\begin{align}
	\begin{split}
		\label{equation:massage one}
		\norm{df}_e^2 &= \int_I\int_I (f(y)-f(x))^2W(x, y) \ dx dy\\
		&= 2\int_I f^2d_W\ d\mu_L - 2\int_I\int_I f(x)f(y)W(x, y)\ dxdy\\
		&= 2 - 2\int_I\int_I f(x)f(y)W(x, y) \ dxdy\\
	\end{split}
\end{align}
A simple computation shows that
\begin{equation}
	\int_I\int_I f(x)f(y)W(x, y)\ dxdy = \int_I\int_I F(x) F(y) W(x, y) \ dxdy
\end{equation}
So from Equation \ref{equation:massage one} we have
\begin{equation}
	\label{equation:end of massage one}
	\norm{df}_e^2 = 2-2\int_I\int_I F(x) F(y) W(x, y) \ dxdy
\end{equation}
Further, since $d_W$ is $\mc A$-measurable, we have
\begin{equation}
	\E[fd_W|\mc A] = d_W\E[f|\mc A] = Fd_W
\end{equation}
Therefore $\int_I Fd_W\ d\mu_L = 0$, that is, $\ab{F, \one}_v=0$.
Also, since $F$ is constant on each member of $\mc P$, we have
\begin{equation}
	\frac{\norm{dF}_e^2}{\norm{F}_v^2} \geq \lambda_G
\end{equation}
provided $F$ is not identically zero.
Therefore, whether or not $F$ is identically zero, we have
\begin{align}
	\begin{split}
		\label{equation:massage two}
		\norm{dF}_e^2 &\geq \lambda_G \norm{F}_v^2\\
		\ra \int_I\int_I (F(y)-F(x))^2 W(x, y) \ dxdy &\geq \lambda_G \int_I F^2d_W\ d\mu_L\\
		\ra 2\int_I F^2d_W\ d\mu_L - 2\int_I\int_I F(x)F(y) W(x, y)\ dxdy &\geq \lambda_G\int_IF^2d_W\ d\mu_L\\
		2- 2\int_I\int_I F(x)F(y) W(x, y)\ dxdy &\geq 2-(2-\lambda_G)\int_IF^2d_W\ d\mu_L\\
	\end{split}
\end{align}
But
\begin{align}
	\begin{split}
		\E[f^2 d_W|\mc A] = d_W\E[f^2|\mc A] \geq d_W\E[f|\mc A]^2 = F^2d_W
	\end{split}
\end{align}
Therefore
\begin{equation}
	1 = \int_If^2d_W\ d\mu_L \geq \int_I F^2d_W\ d\mu_L
\end{equation}
and hence
\begin{equation}
	\lambda_G = 2-(2-\lambda_G) \leq 2 - (2-\lambda_G)\int_IF^2d_W\ d\mu_L
\end{equation}
where we have used the simple fact that $2-\lambda_G$ is non-negative.\footnote{See also the Buser inequality for graphons (Theorem \ref{buser-graphon}).}
Using this in Equation \ref{equation:massage two} we have
\begin{equation}
	2- 2\int_I\int_I F(x)F(y) W(x, y)\ dxdy \geq \lambda_G
\end{equation}
Finally, using this in Equation \ref{equation:end of massage one} we have
\begin{equation}
	\norm{df}_e^2 \geq \lambda_G
\end{equation}
and we are done.

\begin{rmk}
	In conjunction  with Lemma \ref{lemma:comparing cheeger constants} it follows that one can recover Cheeger-Buser type inequalities for regular graphs once we have proven the same for graphons.
\end{rmk}

\section{The Cheeger-Buser Inequalities for Graphons}

\subsection{Convergence of Cheeger constants}
\label{subsection:convergence of cheeger constants} The aim of this subsection is to provide an example of a sequence  of graphons $W_n$ converging to a graphon $W$ such that the corresponding Cheeger constants {\it do not } converge. This preempts the possibility of deducing the Cheeger inequality for graphons directly from that of finite weighted graphs.

A \define{kernel} is a bounded symmetric measurable function $U:I^2\to \R$.
Thus a graphon is nothing but a kernel taking values in the unit interval.
The set of all all kernels $\mc W$ is naturally a vector space over $\R$.
The \define{cut norm} of a kernel $U\in \mc W$ is defined as
\begin{equation}
	\norm{U}_{\Box} = \sup_{A, B\subseteq I}\lrmod{\int_{A\times B} U}
\end{equation}
This makes $\mc W$ into a normed linear space.
Note that the cut norm of a kernel is dominated by the $L^1$ norm with respect to  the Lebesgue measure.

A natural approach to proving the Cheeger-Buser inequalities is the following.
Let $W$ be a graphon and assume for simplicity that the degree of $W$ is bounded away from $0$, i.e, there is $d>0$ such that $d_W\geq d$ $\mu_L$-a.e.
Let $\mc P_n$ be the partition of $I$ defined as
\begin{equation}
	\mc P_n=\set{[0, 1/2^n), [1/2^n, 2/2^n) , \ldots, [(2^n-1)/2^n, 1]}
\end{equation}
and define the partition $\mc Q_n$ of $I^2$ as $\mc Q_n=\set{P\times P':\ P, P'\in \mc P_n}$.
Let $\mc F_n$ be the $\sigma$-algebra on $I$ generated by the partition $\mc Q_n$.
Let $U_n=\E[W|\mc F_n]$ and $H_n$ be the weighted graph on $\set{1 , \ldots, 2^n}$ which gives rise to the graphon $U_n$.
Finally, define $G_n$ as the weighted graph on $\set{1 , \ldots, 2^n}$ obtained by `making $H_n$ loopless', that is, by assigning zero weights to the loops in $H_n$ and keeping all other weights intact.
Let $W_n$ be the graphon corresponding to $G_n$.
Note that $\norm{W_n-U_n}_1\leq 1/2^n$.
By the martingale convergence theorem (\cite[Theorem 5.5]{einsiedler_ward_ergodic_theory}) we have that the sequence $(U_n)$ converges to $W$ in the $L^1$-norm, and hence so does the sequence $(W_n)$.

So we have a sequence $(G_n)$ of loopless weighted graphs such that
\begin{enumerate}
	\item $G_n$ has $2^n$ vertices.
	\item $\vol_\text{max}(G_n)/\vol_\text{min}(G_n)\leq d/2$ for all large enough $n$.
	\item $\norm{W-W_n}_1$, and hence $\norm{W-W_n}_\Box$, approaches $0$ as $n$ approaches $\infty$, where $W_n$ is the graphon corresponding to $G_n$

\end{enumerate}

By Lemma \ref{lemma:comparing cheeger constants} it follows that the Cheeger constant of $W_n$ is a good proxy for the Cheeger constant of $G_n$.
Also, from Section \ref{section:comparing the bottom of the spectrum}, we know that $\lambda_{G_n}=\lambda_{W_n}$. It is shown in \cite{bclsv1,bclsv2} that if $W_n \to W$ in the cut-norm then the bottom of the spectrum of the {\it unnormalized } Laplacian of $W_n$ converges to that of $W$. This suggests a similar convergence result for the normalized  Laplacian at least with a uniform lower bound on the degree $d_W(x)$. 

If we were to try to deduce the Cheeger-Buser inequalities for the graphon $W$
from the classical Cheeger-Buser inequalities for weighted graphs, we would thus need to establish the following:

\begin{quote}
 Let $W_n$ be a sequence of graphons converging to a graphon $W$ in the cut norm.
	Then $h_{W_n}\to h_W$ as $n\to \infty$.
\end{quote}

But the above statement is not necessarily true.
We will in fact give a {\bf counterexample} to  the following statement.

\begin{quote}
 Let $W_n$ be a sequence of graphons converging to a graphon $W$ in the $L^1$-norm.
	Then $h_{W_n}\to h_W$ as $n\to \infty$.
\end{quote}

For each $n$ define a graphon $W_n$ as (see the following figure)

$$
	W_n(x, y) =
	\left\{
	\begin{array}{ll}
		1 & \mbox{if } 0\leq x\leq 1/2-1/n\\
		&\quad 1/2+1/n\leq y\leq 1, \\
		\\
		1 & \mbox{if } 1/2+1/n\leq x\leq 1\\
		&\quad 0\leq y\leq 1/2-1/n\\
		\\
		1 & \mbox{if } 1/2-e^{-n} - 1/n\leq x\leq 1/2+e^{-n} + 1/n\\
		&\quad 1/2 -e^{-n} -1/n \leq y\leq 1/2 + e^{-n}+1/n\\
		0 & \mbox{otherwise}
	\end{array}
	\right.
$$

\begin{figure}[H]
	\centering
	\includegraphics[scale=1.5]{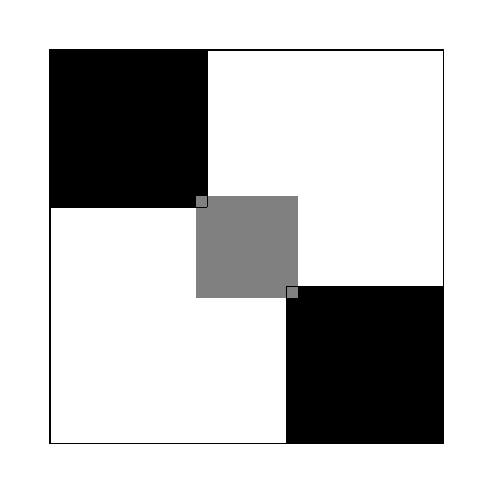}
	\caption{The graphon $W_n$.}
\end{figure}
\noindent
Note that each $W_n$ is connected.
Let $W$ be the graphon corresponding to the complete graph on $2$ vertices.
It is clear that $W_n$ converges to $W$ in the $L^1$-norm.
Let us estimate the Cheeger constant of $W_n$.
Define $A_n$ as the interval $(1/2-e^{-n}-1/n, 1/2+e^{-n}+1/n)$.
Then
\begin{equation}
	h_{W_n}\leq h_{W_n}(A_n) = \frac{2\times e^{-n}\times \frac{2}{n}}
					{
						\frac{2}{n} \times \lrp{ \frac{2}{n} + 2e^{-n}}
					}
\end{equation}
Thus $h_{W_n}\to 0$ as $n\to \infty$.
But $h_W=1/2$ and thus we see that the Cheeger constant of $W_n$ does not converge to that of $W$.
%
%
%

\subsection{Buser Inequality for Graphons}\label{sec:buser-graphon}
\begin{theorem}[Buser Inequality]
	Let $W$ be a connected graphon.
	Then
	\begin{equation}
		\lambda_W\leq 2h_W \text{ and } \lambda_W\leq g_W
	\end{equation}
	\label{buser-graphon}
\end{theorem}
\begin{proof}
	We adapt the proof of Lemma 2.1 in \cite{fan_chung_spectral_graph_theory}.
	Let $A\sqcup B$ form a measurable partition of $I$ with $0<\mu_L(A)< 1$.
	Define $f:I\to \R$ as
	\begin{equation}
		f(x) =
		\twopartdef{ \ds\frac{1}{\vol(A)}}{x\in A}{-\ds\frac{1}{\vol(B)}}{x\in B}
	\end{equation}
	Then $f\in \constone^\perp_v$.
	Now
	\begin{align}
		\begin{split}
			\lambda_W &\leq \frac{\norm{df}^2_e}{\norm{f}^2_v} \\
			&=  \frac{\int_E (f(x)-f(y))^2 W(x, y)\ dydx}{\int_0^1 f(x)^2 d_W(x)\ dx}\\
			&=  \frac{\int_{I\times I} (f(x)-f(y))^2 W(x, y)\ dydx}{2\int_0^1 f(x)^2 d_W(x)\ dx}\\
			&=\frac{\int_{A\times B} (f(x)-f(y))^2 W(x, y)\ dydx + \int_{B\times A} (f(x)-f(y))^2 W(x, y)\ dydx}{2\lrb{\int_A f(x)^2 d_W(x)\ dx + \int_B f(x)^2d_W(x)\ dx}}\\
			&=  \frac{\lrp{ \frac{1}{\vol(A)} + \frac{1}{\vol(B)}}^2 \lrp{\int_{A\times B} W(x, y)\ dxdy + \int_{B\times A} W(x, y)\ dxdy}}{2\lrb{\frac{1}{\vol(A)^2}\vol(A) + \frac{1}{\vol(B)^2}\vol(B)}}\\
			&= \lrp{\frac{1}{\vol(A)} + \frac{1}{\vol(B)}} \int_{A\times B} W(x, y)\ dxdy\\
			&\leq  2 \frac{\int_{A\times B} W}{ \min\set{\vol(A), \vol(B)}}
		\end{split}
	\end{align}
	Since $B=A^c$, and since the above holds for all choices of $A$ with $0< \mu_L(A)< 1$, we have $\lambda_W\leq 2h_W$.
	From the penultimate inequality above we also get
	\begin{equation}
		\lambda_W\leq \ds \frac{ \int_{A\times B} W(x, y)\ dxdy}{\vol(A)\vol(B)}
	\end{equation}
	since $\vol(A)+\vol(B)\leq 1$.
	This leads to $\lambda_W\leq g_W$.
\end{proof}

\subsection{The Co-area Formula for Graphons}

Consider a finite graph $G=(V ,E)$ and let $f:V\to \R$ be any map.
Orient the edges of $G$ in such a way that for each oriented edge $e$ we have $f(e^+)\geq f(e^-)$.
Let $\gamma_0< \gamma_1 < \cdots < \gamma_k$ be all the reals in the image of $f$.
Define $S_i=\set{v\in V:\ f(v)\geq \gamma_i}$.
Then we have
\begin{equation}
	\label{baby coarea formula for graphs}
	\sum_{e\in E} df(e) = \sum_{i=1}^m (\gamma_i-\gamma_{i-1}) |E(S_i^c, S_i)|
\end{equation}
where $E(S_i^c, S_i)$ denotes the set of all the edges in $G$ which have their tails in $S_i^c$ and heads in $S_i$.
To see why Equation \ref{baby coarea formula for graphs} is true, we fix an edge $e$ and see how much it contributes to the sum on the RHS.
We add $\gamma_i-\gamma_{i-1}$ for each $i$ such that $e^-\in S_i^c$ and $e^+\in S_i$.
This adds up to a total of $df(e)$, which is the same as the contribution of $e$ to the LHS.

If $G$ were a weighted graph with weight function $w:E\to \R^+$, Equation \ref{baby coarea formula for graphs} takes the form
\begin{equation}
	\label{weighted coarea formula for graphs}
	\sum_{e\in E} df(e)w(e) = \sum_{i=1}^m (\gamma_i-\gamma_{i-1}) e_w(S_i^c, S_i)
\end{equation}
where $e_w(S_i^c, S_i)$ denotes the sum of weights of all the edges which have their tails in $S_i^c$ and heads in $S_i$.

Let us see how Equation \ref{weighted coarea formula for graphs} generalizes for graphons.
Let $W$ be a graphon and $f:I\to \R$ be in $L^2(I, \nu)$.
Define $E_f$ to be the set $\set{(x, y)\in I^2:\ f(y)> f(x)}$.
Let $S_t$ denote the set $f^{-1}(t, \infty)$.
Then
\begin{equation}
	\label{useless coarea formula for graphons}
	\int_{E_f} df(x, y) W(x, y)\ dxdy = \int_{-\infty}^\infty e_W(S_t^c, S_t)\ dt
\end{equation}
This can be easily proved using Fubini's theorem.
We shall however need a slight variant of this formula in order to establish Cheeger's inequality.

\begin{theorem}[Co-area formula for graphons]\label{coarea-graphon}
	Let $W$ be a graphon and $f:I\to \R$ be an arbitrary map in $L^2(I, \nu)$.
Define $f_+:I\to \R$ and $f_-:I\to \R$ as $f_+=\max\set{f, 0}$ and $f_-=-\min\set{f, 0}$.
Let $S_t=f^{-1}(t, \infty)$.
Then
\begin{align}
	\label{coarea formula for graphons}
	\begin{split}
		\int_{E_f} |df^2_+|W &= \int_0^\infty e_W(S_{\sqrt t}^c, S_{\sqrt t})\ dt = \int_0^\infty 2te_W(S_t^c, S_t)\ dt,\quad \text{and}\\
		\int_{E_f} |df_-^2|W &= \int_{0}^\infty e_W(S_{-\sqrt t}^c, S_{-\sqrt t}) = \int_0^\infty 2te_W(S_{-t}^c, S_{-t})\ dt
	\end{split}
\end{align}
\end{theorem}
\begin{proof}
	We prove the first one.
The second one is similar.
We have by change of variables that
\begin{equation}
	\int_0^\infty e_W(S_{\sqrt t}^c, S_{\sqrt t})\ dt = \int_0^\infty 2te_W(S_t^c, S_t)\ dt
\end{equation}
Now
\begin{align}
	\begin{split}
		\int_0^\infty 2te_W(S_t^c, S_t)\ dt &= \int_0^\infty 2t\lrb{\int_{S_t^c\times S_t} W(x, y)\ dx dy}dt\\
		&= \int_0^\infty\lrb{\int_{I^2} 2t \chi_{S_t^c\times S_t}(x, y) W(x, y)\ dxdy} dt\\
		&= \int_{I^2}\lrb{\int_0^\infty 2t\chi_{S_t^c\times S_t}(x, y) W(x, y) \ dt} dxdy\\
		&= \int_{I^2}\lrb{\int_0^\infty 2t\chi_{S_t^c\times S_t}(x, y)\ dt} W(x, y) dxdy\\
		&= \int_{E_f}\lrb{\int_0^\infty 2t\chi_{S_t^c\times S_t}(x, y)\ dt} W(x, y) dxdy\\
		&= \int_{E_f}\lrb{\int_{f_+(x)}^{f_+(y)} 2t\ dt} W(x, y) dxdy\\
		&= \int_{E_f} (f_+^2(y)-f_+^2(x))W(x, y)\ dxdy\\
		&= \int_{E_f} |df^2_+|W
	\end{split}
\end{align}
as desired.
\end{proof}

\subsection{Cheeger's Inequality for Graphons}

In this subsection we will prove the following.

\begin{theorem}
	\label{cheeger-graphon}
	Let $W$ be a connected graphon.
	Then
	\begin{equation}
		\lambda_W\geq \frac{h^2_W}{8}
	\end{equation}
\end{theorem}

Before we prove  Cheeger's inequality above, we first obtain a more convenient formula (Lemma \ref{denseapp} below) for $\lambda_W$.
%
%
Consider the map $\mc I:L^2(I, \nu)\to \R$ defined as
\begin{equation}
	\mc I(f) = \int_0^1 f(x)d_W(x)\ dx = \ab{f, \constone}_v
\end{equation}
We show that $L^\infty(I)\cap \constone^\perp_v$ is dense in $\constone^\perp_v$.
Let $P:L^2(I, \nu)\to L^2(I, \nu)$ be the map defined as $P(f)= f-\mc I(f)$.
Then $P$ is a bounded linear operator.
Also,  we have
\begin{equation}
	P^2(f)=P(f-\mc I(f))=f-\mc I(f)=P(f)
\end{equation}
So $P^2=P$.
Further,
\begin{align}
	\begin{split}
		\ab{Pf, g}_v -\ab{f, Pg}_v &= \ab{f-\mc I(f), g}_v-\ab{f, g-\mc I(g)}_v\\
		&= -\ab{\mc I(f), g}_v + \ab{f, \mc I(g)}_v\\
		&= -\mc I(f)\mc I(g) + \mc I(f)\mc I(g)\\
		&= 0
	\end{split}
\end{align}
Therefore $P$ is self-adjoint.
This means that $P$ is the orthogonal projection onto its image.
It is clear that $\im(P)\subseteq \constone^\perp_v$, and also that $P$ behaves as the identity when restricted to $\constone^\perp_v$.
Therefore $P$ is the orthogonal projection onto $\constone^\perp_v$.
It is also clear that $P(L^\infty(I))\subseteq L^\infty(I)$. 
 We conclude that $L^\infty(I)\cap \constone^\perp_v$ is dense in $\constone^\perp_v$.

Now let $g\in \constone^\perp_v$ be an arbitrary nonzero vector.
Then both $\norm{dg}_e$ and $\norm{g}_v$ are nonzero.\footnote{If $\norm{dg}_e$ were equal to $0$ then $g$ would be an essentially constant function, which would force  $g=0$ since $g\in \constone^\perp_v$.}
Let $M>0$ be such that $\norm{dg}_e, \norm{g}_v\geq M$.
Let $\varepsilon>0$ be arbitrary and choose $g'\in L^\infty(I)\cap \constone^\perp_v$ such that $\norm{g-g'}_v< \varepsilon M$.
We had shown in the proof of Lemma \ref{d-cont}, $\norm{df}_e\leq 2\norm{f}_v$ for all 
 $f\in L^2(I, \nu)$.  So  $$\norm{d(g-g')}_e=\norm{dg-dg'}_e< 2\varepsilon M.$$
Thus we have
\begin{equation*}
	|\norm{g}_v-\norm{g'}_v|< \varepsilon M \text{ and } |\norm{dg}_e - \norm{dg'}_e| < 2\varepsilon M
\end{equation*}
Hence we can approximate $\norm{dg}_e/\norm{g}_v$ arbitrarily well by the expressions of the form $\norm{dg'}_e/\norm{g'}_v$ by choosing a suitable $g'\in L^\infty(I)\cap \constone^\perp_v$.
We have proved
\begin{lemma}\label{denseapp}
	Let $W$ be a connected graphon.
	Then
\begin{equation}
	\label{rayleigh quotient approximation}
	\lambda_W = \inf_{g\in \constone^\perp_v:\ g\neq 0} \frac{\norm{dg}_e^2}{\norm{g}_v^2} =  \inf_{\substack{g\in \constone^\perp_v:\ g\neq 0,\\ g\in L^\infty(I)}} \frac{\norm{dg}_e^2}{\norm{g}_v^2}
\end{equation}
\end{lemma}

Now we are ready to prove Theorem \ref{cheeger-graphon}.
Let $g:I\to \R$ be an arbitrary map in  $L^\infty(I)$ with $\norm{g}_v=1$ and $\ab{g, \constone}_v=0$.
To prove Theorem \ref{cheeger-graphon} it suffices to show that $\norm{dg}_e^2\geq \frac{1}{8} h^2_W$.
Let
\begin{equation}
	t_0=\sup\set{t\in \R:\ \vol_W(g^{-1}(-\infty, t))\leq \frac{1}{2}\vol_W(I)}
\end{equation}
The number $t_0$ exists since $g$ is $L^\infty$.
Define $f=g-t_0$.
Then both the sets $\set{f < 0}$ and $\set{f>0}$ have volumes at most half of $\vol_W(I)$.
Also
\begin{equation}
	\norm{f}_v^2=\norm{g-t_0}_v^2 = \norm{g}_v^2+ \norm{t_0}_v^2-2t_0\ab{g, \constone}_v = 1+\norm{t_0}_v^2\geq 1
\end{equation}
Clearly, $df=dg$.
Therefore
\begin{equation}
	\label{passing from g to f in graphons}
	\norm{dg}_e^2\geq \frac{\norm{df}_e^2}{\norm{f}_v^2}
\end{equation}
\begin{lemma}
	\label{passing to squares in graphons}
	\begin{equation}
		\norm{df}_e^2
		\geq\frac{1}{8\norm{f}_v^2} \lrb{\int_{E_f} |df^2_+|W + \int_{E_f} |df^2_-|W}^2
	\end{equation}
\end{lemma}
\begin{proof}
	Note that
	\begin{equation}
		\label{passing to edges oriented by f-graphon} 
		\norm{df}_e^2= \int_{E} |df|^2W = \int_{E_f} |df|^2W 
	\end{equation}
	where $E_f=\set{(x, y)\in I^2:\ f(y)> f(x)}$.
	Also
	\begin{equation}
		\label{split in two parts in graphons}
		\int_{E_f} |df|^2W \geq \int_{E_f} |df_+|^2W + \int_{E_f} |df_-|^2W
	\end{equation}
	This is because $|df|^2\geq |df_+|^2+ |df_-|^2$ is true pointwise in $E_f$.
	By Cauchy-Schwarz we have
	\begin{align}
		\begin{split}
			\lrb{\int_{E_f} |df_+|^2W}^{1/2} \lrb{\int_{I^2} (|f(x)| + |f(y)|)^2W(x, y)\ dxdy}^{1/2}
			&\geq \int_{E_f} |df^2_+|W\\
		\end{split}
	\end{align}
	Using Cauchy-Schwarz again, we can show that
	\begin{equation}
		4\norm{f}_v^2\geq \int_{I^2} (|f(x)| + |f(y)|)^2W(x,y)\ dxdy
	\end{equation}
	which gives
	\begin{align}
		\begin{split}
			2\norm{f}_v\lrb{\int_{E_f} |df_+|^2W}^{1/2} &\geq \int_{E_f} |df_+^2|W\\
			\ra \int_{E_f}|df_+|^2W &\geq \frac{1}{4\norm{f}_v^2} \lrp{\int_{E_f} |df_+^2|W}^2
		\end{split}
	\end{align}
	Similarly,
	\begin{equation}
		\int_{E_f} |df_-|^2W \geq \frac{1}{4\norm{f}_v^2} \lrp{\int_{E_f} |df_-^2|W}^2
	\end{equation}
	Using these in Equation \ref{split in two parts in graphons} gives
	\begin{align}
		\begin{split}
			\int_{E_f} |df|^2W &\geq \frac{1}{4\norm{f}_v^2}\lrb{\lrp{\int_{E_f} |df_+^2|W}^2 + \lrp{\int_{E_f} |df_-^2|W}^2}\\
		&\geq \frac{1}{8\norm{f}_v^2} \lrb{\int_{E_f} |df^2_+|W + \int_{E_f} |df^2_-|W}^2
		\end{split}
	\end{align}
	and we have proved the lemma.
\end{proof}

We now proceed to complete the proof of Theorem \ref{cheeger-graphon}.\\

\noindent {\bf Proof of Theorem \ref{cheeger-graphon}:}\\
The \hyperref[coarea formula for graphons]{Co-area Formula} Theorem \ref{coarea-graphon} gives
\begin{equation}
		\int_{E_f} |df^2_+|W = \int_0^\infty 2te_W(S_t^c, S_t)\ dt\quad \text{ and }\quad
		\int_{E_f} |df_-^2|W = \int_{0}^\infty 2te_W(S_{-t}^c, S_{-t})\ dt
\end{equation}
But
\begin{align}
	\begin{split}
		\int_0^\infty 2te_W(S_t^c, S_t) \ dt &\geq h_W\int_0^\infty 2t\vol(S_t)\ dt\\
		&=h_W\int_0^\infty 2t\lrb{\int_{I^2} \chi_{I\times S_t}(x, y) W(x, y)\ dx dy}dt\\
		&=h_W\int_{I^2}\lrb{\int_0^\infty 2t\chi_{I\times S_t}(x, y) \ dt} W(x, y)dxdy\\
		&=h_W\int_{I^2}\lrb{\int_0^\infty 2t\chi_{I\times S_t}(x, y)\ dt} W(x, y)dxdy\\
		&=h_W\int_{I^2}\lrb{\int_0^{f_+(y)} 2t \ dt} W(x, y)\ dxdy\\
		&=h_W\int_{I^2} f_+^2(y) W(x, y)\ dx dy
	\end{split}
\end{align}
Similarly
\begin{align}
	\begin{split}
		\int_0^\infty 2te_W(S_{-t}^c, S_{-t}) \ dt &\geq h_W\int_0^\infty 2t\vol(S_{-t}^c)\ dt\\
		&=h_W\int_0^\infty 2t\lrb{\int_{I^2} \chi_{I\times S_{-t}^c}(x, y) W(x, y)\ dx dy}dt\\
		&=h_W\int_{I^2}\lrb{\int_0^\infty 2t\chi_{I\times S_{-t}^c}(x, y)\ dt} W(x, y) dxdy\\
		&=h_W\int_{I^2}\lrb{\int_0^\infty 2t\chi_{I\times S_{-t}^c}(x, y)\ dt}W(x, y)dxdy\\
		&=h_W\int_{I^2}\lrb{\int_0^{f_-(y)} 2t \ dt} W(x, y)dxdy\\
		&=h_W\int_{I^2} f_-^2(y) W(x, y)\ dx dy
	\end{split}
\end{align}
Therefore
\begin{align}
	\begin{split}
		\int_{E_f} |df^2_+|W + \int_{E_f}|df^2_-|W &\geq h_W \lrb{\int_{I^2} f_+^2(y) W(x, y)\ dxdy+ \int_{I^2} f_-^2(y)W(x, y)\ dxdy}\\
		&= h_W \lrb{\int_{I^2} (f_+^2(y)+ f_-^2(y)) W(x, y)\ dxdy}\\
		&= h_W \lrb{\int_{I^2} f^2(y)W(x, y)\ dxdy}\\
		&=h_W \lrb{\int_If^2(y) d_W(y)\ dy} = h_W \norm{f}^2_v
	\end{split}
\end{align}
Combining this with Lemma \ref{passing to squares in graphons}, we have
\begin{equation}
	\norm{df}_e^2 \geq \frac{1}{8\norm{f}_v^2}  h^2_W\norm{f}_v^4
\end{equation}
and thus
\begin{equation}
	\frac{\norm{df}_e^2}{\norm{f}_v^2} \geq \frac{1}{8} h^2_W
\end{equation}
Lastly, using Equation \ref{passing from g to f in graphons} we have
\begin{equation}
	\norm{dg}_e^2\geq \frac{1}{8}h^2_W
\end{equation}
and we are done.\hfill $\blacksquare$


\section{Cheeger Constant for Graphings and the Cheeger-Buser Inequalities}\label{graphing} We now turn to graphings.
For the purposes of this section $G=(I, \mu, E)$ will denote a  graphing.
As discussed  in Section \ref{prel-graphing}, graphings are substantially different from graphons in terms of their structure.
In spite of this difference,
Lemma \ref{what eta really is} will allow us to furnish proofs that are, at least at a  formal level, extremely similar to the proofs in Section \ref{graphon} above. However, the actual intuition and idea behind the proofs will really go back to Theorem \ref{strand theorem}. In this section, we shall therefore try to convey to the reader both the formal similarity with the proofs in Section \ref{graphon} as well as  the actual structural idea  going back to Theorem \ref{strand theorem}.
\subsection{Cheeger Constant for Graphings}\label{sec:cheeger-graphing}
For two measurable subsets $A$ and $B$ of $I$, we define
\begin{equation}
	e_G(A, B) = \eta(A, B)= \int_A \deg_B(x)\ d \mu(x)
\end{equation}
For a measurable subset $A$ of $I$, the \define{volume} of $G$ over $A$ is defined as
\begin{equation}
	\vol(A) = \int_{A} \deg(x) \ d\mu(x) = e_G(A, I)
\end{equation}
A graphing is said to be \define{connected} if for all measurable subsets $A$ of $I$ with $0< \mu(A)< 1$ we have $e_G(A, A^c) \neq 0$.
Note that if $G$ is connected then $\deg>0$ a.e.

Given a graphing $G$, we define the \define{Cheeger constant} of $G$ as
\begin{equation}
	h_G= \inf_{A\subseteq I:\ 0< \mu(A)< 1}\frac{e_G(A, A^c)}{\min\set{\vol(A), \vol(A^c)}}
\end{equation}
A  symmetrized version of the above constant which we will be referred to as the \define{symmetric Cheeger constant} is defined as
\begin{equation}
	g_G = \inf_{A\subseteq I:\ 0< \mu(A)< 1} \frac{e_G(A, B)}{\vol(A)\vol(A^c)}
\end{equation}
Note that the above defined constants exist for connected graphings.

%
%
%

%
\subsection{Buser Inequality for Graphings}

We first observe that the multiplicity of the singular value $0$ of the Laplacian of a connected graphing $(I, \mu, E)$ is 1.
For $f\in L^2(I, \mu)$, we have $\lap f=0$ if and only if $df=0$.
As in the case of graphons it now suffices to show the following: $df=0$ if and only if $f$ is constant (up to a set of measure zero).
Of course,  $df=0$ for constant $f$.
Conversely, assume that $df=0$.
Then
\begin{equation}
	\int_{E^+} (df)^2\ d\eta =0 
\end{equation}
which implies that

\begin{equation}
	\int_{E^+} (f(y)-f(x))^2\ d\eta(x, y)=0.
\end{equation}
Since
\begin{equation}
	\int_{E^+} (f(y)-f(x))^2\ d\eta(x, y)=\int_{E^-} (f(y)-f(x))^2\ d\eta(x, y).
\end{equation}
it follows that
\begin{equation}
	\int_{I^2} (f(y)-f(x))^2\ d\eta(x, y)=0
\end{equation}
For each $t\in \R$, let $S_t=f^{-1}(t, \infty)$.
Therefore,
\begin{equation}
	\int_{S_t^c\times S_t} (f(y)-f(x))^2\ d\eta(x, y)=0.
\end{equation}
It follows  that $(f(y)-f(x))^2$ is $\eta$-a.e. $0$ on $S_t^c\times S_t$.
But $f(y)-f(x)\neq 0$ for all $(x, y)\in S_t\times S_t^c$.
So $\eta(S_t^c, S_t)=0$  for all $t$.
The connectedness of $G$ then implies that either $S_t$ or $S_t^c$ has $\mu$-measure $0$.
So our claim follows from the following lemma whose proof is an exact replica of Lemma \ref{ess-const-graphon} and we omit it.

\begin{lemma}
	Let $f:I\to \R$ be a measurable function such that for all $t\in \R$ we have either $f^{-1}(-\infty, t]$ or $f^{-1}(t, \infty)$ has $\mu$-measure $0$.
	Then $f$ is constant $\mu$-a.e.
\end{lemma}

The eigenfunctions corresponding to $0$ are thus  the essentially constant functions:
 the $0-$eigenspace of $\lap$  is generated by $\constone$.
Define
\begin{equation}
	\lambda_G = \inf_{g\in \constone^\perp_v:\ g\neq 0} \frac{\ab{g, \lap g}_v}{\ab{g, g}_v} = \inf_{g\in \constone^\perp_v:\ g\neq 0} \frac{\norm{dg}^2_e}{\norm{g}^2_v}
\end{equation}
(Here, $\constone^\perp_v$ denotes the orthogonal complement of $\constone$ with respect to the inner product $\ab{\cdot, \cdot}_v$).

\begin{theorem}[Buser Inequality]
	Let $G=(I, \mu, E)$ be a connected graphing.
	Then
	\begin{equation}
		\lambda_G\leq 2h_G \text{ and } \lambda_G\leq g_G
	\end{equation}
	\label{buser-graphing}
\end{theorem}

The proof below exploits the fact that one can decompose a graphing into finitely many matchings (Theorem \ref{strand theorem}).  An alternate proof can also be given following that of 
Theorem \ref{buser-graphon} replacing $W(x,y)dxdy$ formally with $d\eta(x, y)$.
\begin{proof}
	Let $\vp_i:C_i\to C_i$, $1\leq i\leq k$, be $\mu$-measure preserving involutions such that
	\begin{equation}
		E= \bigsqcup_{i=1}^k \set{(x, \vp_i(x)):\ x\in C_i}
	\end{equation}
	Extend each $\vp_i$ to a map $\vp_i:I\to I$ by declaring $\vp_i(x)=x$ for all $x\notin C_i$.
	Let $A\sqcup B$ form a measurable partition of $I$ with $0<\mu(A)< 1$.
	Define $f:I\to \R$ as
	\begin{equation}
		f(x) =
		\twopartdef{ \ds\frac{1}{\vol(A)}}{x\in A}{-\ds\frac{1}{\vol(B)}}{x\in B}
	\end{equation}
	Then $f\in \constone^\perp_v$.
	Now
	\begin{align*}
		\begin{split}
			\lambda_G &\leq \frac{
									\norm{df}^2_e
								}{
									\norm{f}^2_v
								} =  \frac{
						\int_{E^+} (f(x)-f(y))^2 \ d\eta(x, y)
					}{
						\int_0^1 f(x)^2 \deg(x)\ d\mu(x)
					}
			=  \frac{
						\int_{I^2} (f(x)-f(y))^2 \ d\eta(x, y)
					}{
						2\int_0^1 f(x)^2 \deg(x)\ d\mu(x)
					}\\ 
			&= \frac{
						\sum_{i=1}^k \int_{I} (f(x)-f(\vp_i(x)))^2\ d\mu(x)
					}{
						2\int_0^1 f(x)^2\deg(x)\ d\mu(x)
					}\\ 
			&= \frac{
						\sum_{i=1}^k\lrb{\int_A (f(x)-f(\vp_i(x)))^2\ d\mu(x) + \int_B (f(x)-f(\vp_i(x)))^2\ d\mu(x)}
					}
					{
						2\int_0^1 f(x)^2\deg(x)\ d\mu(x)
					}\\ 
			&= \frac{
						\sum_{i=1}^k \lrb{\int_A(f(x) - f(\vp_i(x)))^2\chi_B(\vp_i(x))\ d\mu(x) - \int_B (f(x)-f(\vp_i(x)))^2\chi_A(\vp_i(x))\ d\mu(x)}
					}
					{
							2\int_0^1 f(x)^2 \deg(x)\ d\mu(x)
					}\\ 
			&= \frac{
						\sum_{i=1}^k \lrb{\int_{A} \lrp{ \frac{1}{\vol(A)} + \frac{1}{\vol(B)}}^2\chi_B(\vp_i(x))\ d\mu(x) + \int_{B} \lrp{ \frac{1}{\vol(B)} + \frac{1}{\vol(A)}}^2\chi_A(\vp_i(B)) \ d\mu(x)}
					}
					{
						2\int_0^1 f(x)^2 \deg(x)\ d\mu(x)
					}\\ 
			&= \frac{
					\lrp{ \frac{1}{\vol(A)} + \frac{1}{\vol(B)}}^2 \sum_{i=1}^k\lrb{ \int_A \chi_B(\vp_i(x))\ d\mu(x) + \int_B\chi_A(\vp_i(x))\ d\mu(x)}
					}
					{
						2\lrb{\int_A f(x)^2 \deg(x)\ d\mu(x) + \int_B f(x)^2\deg(x)\ d\mu(x)}
					}\\ 
			&= \frac{
						\lrp{ \frac{1}{\vol(A)} + \frac{1}{\vol(B)}}^2 \lrb{ \int_A \deg_B(x)\ d\mu(x) + \int_B\deg_A(x)\ d\mu(x)}
					}
					{
						2\lrb{\frac{1}{\vol(A)^2}\vol(A) + \frac{1}{\vol(B)^2}\vol(B)}
					}\\ 
			&= \lrp{ \frac{1}{\vol(A)} + \frac{1}{\vol(B)}}e_G(A, B)\\ 
			&\leq  2 \frac{e_G(A, B)}{ \min\set{\vol(A), \vol(B)}}, \frac{e_G(A, B)}{\vol(A)\vol(B)}
		\end{split}
	\end{align*}
	Since $B=A^c$, and since the above holds for all choices of $A$ with $0< \mu(A)< 1$, we have $\lambda_G\leq 2h_G$ and $\lambda_G\leq g_G$.
\end{proof}
\noindent {\bf Alternate Proof:}
	In the proof of 
	Theorem \ref{buser-graphon} replace  $W(x,y)dxdy$ formally with $d\eta(x, y)$.
\hfill $\blacksquare$

\subsection{Co-area Formula for Graphings}

Let $G=(I, E, \mu)$ be a graphing and $f:I\to \R$ be any $L^\infty$-map.
Let $E_f$ be defined as
\begin{equation}
	E_f=\set{(x, y)\in E:\ f(y)> f(x)}
\end{equation}
The set $E_f$ will be referred to as the $f$-\define{oriented edges}  of $G$.
Let $S_t$ denote the set $f^{-1}(t, \infty)$.
Then
\begin{equation}
	\label{useless coarea formula for graphings}
	\int_{E_f} df d\eta = \int_{-\infty}^\infty e_G(S_t^c, S_t)\ dt
\end{equation}
Let us see the proof in the special case when $E$ is given by a single measure preserving involution $\vp:A\to A$ where $A$ is a measurable subset of $I$.
Define $R=\set{x\in A:\ (x, \vp(x))\in E_f}$.
Then the RHS of the above equation is
\begin{align}
	\begin{split}
		\int_{-\infty}^\infty e_G(S_t^c, S_t)\ dt &= \int_{-\infty}^\infty \int_{S_t^c} \deg_{S_t}(x)\ d\mu(x)\ dt
		= \int_{-\infty}^\infty \int_0^1 \chi_{S_t^c}(x) \deg_{S_t}(x)\ d\mu(x)\ dt\\
		&=\int_0^1 \int_{-\infty}^\infty \chi_{S_t^c}(x) \deg_{S_t}(x) \ dt\ d\mu(x)
		= \int_0^1 \int_{f(x)}^\infty \deg_{S_t}(x) \ dt\ d\mu(x)\\
		&= \int_R \int_{f(x)}^\infty \deg_{S_t}(x)\ dt\ d\mu(x)
		= \int_R (f\circ \vp(x)-f(x)) \ d\mu(x)
	\end{split}
\end{align}
On the other hand the LHS of Equation \ref{useless coarea formula for graphings} is
\begin{align}
	\begin{split}
		\int_{E_f} df\ d\eta = \int_{I^2} df\chi_{E_f}\ d\eta &= \int_I \sum_y df(x, y) \chi_{E_f}(x, y) \ d\mu(x)\\
		&= \int_R df(x, \vp(x))\ d\mu(x)\\
		&= \int_R (f\circ \vp(x)-f(x))\ d\mu(x)
	\end{split}
\end{align}
and therefore Equation \ref{useless coarea formula for graphings} holds.
Just as in the case of graphons, we need a slightly different lemma

\begin{theorem}\label{coarea-graphing}
	Let $G=(I, E, \mu)$ be a graphing and $f:I\to \R$ be an arbitrary $L^2$-map.
Define $f_+:I\to \R$ and $f_-:I\to \R$ as the map $f_+=\max\set{f, 0}$ and $f_-=-\min\set{f, 0}$.
Let $S_t=f^{-1}(t, \infty)$.
Then
\begin{align}
	\label{coarea formula for graphings}
	\begin{split}
		\int_{E_f} |df^2_+|\ d\eta &= \int_0^\infty e_G(S_{\sqrt t}^c, S_{\sqrt t})\ dt = \int_0^\infty 2te_G(S_t^c, S_t)\ dt,\quad \text{and}\\
		\int_{E_f} |df_-^2|\ d\eta &= \int_0^\infty e_G(S_{-\sqrt t}^c, S_{-\sqrt t}) = \int_0^\infty 2te_G(S_{-t}^c, S_{-t})\ dt
	\end{split}
\end{align}
\end{theorem}
\begin{proof}
We prove only the first one.
Further, as in the proof of Lemma \ref{what eta really is} (see also Equation \ref{what is eta in terms of mu}) we first assume  that the edge set $E$ is determined by a single $\mu$-measure preserving involutions $\vp:A\to A$ defined on a measurable subset $A$ of $I$.
Define $E_f=\set{(x, y)\in E:\ f(y)> f(x)}$.
Let $R=\set{x\in I:\ (x, \vp(x))\in E_f}$.
We have by change of variables that
\begin{equation}
	\int_0^\infty e_G(S_{\sqrt t}^c, S_{\sqrt t})\ dt = \int_0^\infty 2te_G(S_t^c, S_t)\ dt
\end{equation}
Now
\begin{align}
	\begin{split}
		\int_0^\infty 2te_G(S_t^c, S_t)\ dt &= \int_0^\infty 2t\eta(S_t^c\times S_t)\ dt
		= \int_0^\infty 2t \lrb{\int_0^1\chi_{S_t^c}(x)\deg_{S_t}(x)\ d\mu(x)} dt\\
		&=\int_0^\infty \lrb{\int_0^1 2t\chi_{S_t^c}(x)\deg_{S_t}(x)\ d\mu(x)} dt\\
		&=\int_0^1 \lrb{\int_0^\infty 2t\chi_{S_t^c}(x)\deg_{S_t}(x)\ dt} d\mu(x)\\
		&=\int_0^1 \lrb{\int_{f_+(x)}^\infty 2t\deg_{S_t}(x)\ dt} d\mu(x)\\
		&=\int_R \lrb{\int_{f_+(x)}^\infty 2t\deg_{S_t}(x)\ dt} d\mu(x)\\
		&=\int_R \lrb{\int_{f_+(x)}^{f_+(\vp(x))} 2t\ dt} d\mu(x)\\
		&=\int_R (f_+^2(\vp(x)) - f_+^2(x)) d\mu(x)
	\end{split}
\end{align}
On the other hand
\begin{align}
	\begin{split}
		\int_{E_f}|df_+^2|\ d\eta &= \int_{I^2}|df_+^2(x, y)|\chi_{E_f}(x, y)\ d\eta(x, y)\\
		&= \int_I \sum_y |df_+^2(x, y)|\chi_{E_f}(x, y)\ d\mu(x)\\
		&= \int_R |df_+^2(x, \vp(x))|\ d\mu(x)\\
		&= \int_R (f_+^2(\vp(x)) - f_+^2(x))\ d\mu(x),
	\end{split}
\end{align}
completing the proof for the special case of a single strand.

We now deal with the general case where there may be multiple strands.
Let $\vp_i:A_i\to A_i$, $1\leq i\leq k$, be $\mu$-measure preserving involutions such that $E=\bigsqcup_{i=1}^k \set{(x, \vp_i(x)):\ x\in A_i}$.
Let $G_i$ be the graphing corresponding to $\vp_i$.
So $G=\bigsqcup_{i=1}^k G_i$.
Then
\begin{align}
	\begin{split}
		\int_0^\infty 2te_G(S_t^c, S_t)\ dt &= \sum_{i=1}^k \int_0^\infty 2te_{G_i}(S_t^c, S_t)\ dt
	\end{split}
\end{align}
and
\begin{align}
	\begin{split}
		\int_{E_f} |df_+^2|\ d\eta &= \sum_{i=1}^k \int_{E_f^i} |df_+^2|\ d\eta_i
	\end{split}
\end{align}
where $E_f^i$ are the $f$-oriented edges of $G_i$ and $\eta_i$ is the edge measure of $G_i$.
Thus the general case follows from the special case.
\end{proof}

\subsection{Cheeger Inequality for Graphings}

In this subsection we will prove the following Cheeger inequality for graphings.

\begin{theorem}
	\label{cheeger-graphing}
	Let $G$ be a connected graphing.
	Then
	\begin{equation}
		\lambda_G\geq \frac{h^2_G}{8}
	\end{equation}
\end{theorem}

The proof of Lemma \ref{denseapp} goes through mutatis mutandis to give:
\begin{lemma}\label{denseapp-graphing}
	Let $G$ be a connected graphing.
	Then
\begin{equation}
	\label{rayleigh quotient approximation for graphing}
	\lambda_G = \inf_{g\in \constone^\perp_v:\ g\neq 0} \frac{\norm{dg}_e^2}{\norm{g}_v^2} =  \inf_{\substack{g\in \constone^\perp_v:\ g\neq 0,\\ g\in L^\infty(I, \mu)}} \frac{\norm{dg}_e^2}{\norm{g}_v^2}
\end{equation}
\end{lemma}

We proceed with the proof of Theorem \ref{cheeger-graphing} for graphings.
Let $g:I\to \R$ be an arbitrary $L^\infty$-map with $\norm{g}_v=1$ and $\ab{g, \constone}_v=0$.
To prove Cheeger's inequality it is enough to show that $\norm{dg}_e^2 \geq \frac{1}{8} h^2_G$.
Let
\begin{equation}
	t_0 = \sup\set{t\in \R:\ \vol_G(g^{-1}(-\infty, t))\leq \frac{1}{2}\vol_G(I)}
\end{equation}
and define $f=g-t_0$.
Then both the sets $\set{f < 0}$ and $\set{f>0}$ have volumes at most half of $\vol_G(I)$.
Also
\begin{equation}
	\norm{f}_v^2=\norm{g-t_0}_v^2 = \norm{g}_v^2+ \norm{t_0}_v^2-2t_0\ab{g, \constone}_v = 1+\norm{t_0}_v^2\geq 1
\end{equation}
Clearly, $df=dg$.
Therefore
\begin{equation}
	\label{passing from g to f}
	\norm{dg}_e^2\geq \frac{\norm{df}_e^2}{\norm{f}_v^2}
\end{equation}

\begin{lemma}\label{passing to squares in graphings}
	\begin{equation}
		\label{passing to squares}
		\int_{E_f} |df|^2\ d\eta 
		\geq \frac{1}{8\norm{f}_v^2} \lrb{\int_{E_f} |df^2_+|\ d\eta + \int_{E_f} |df^2_-|\ d\eta}^2
	\end{equation}
\end{lemma}
\begin{proof} As in Lemma \ref{passing to squares in graphons}, we start by observing that
	\begin{equation}
	\label{passing to edges oriented by f} 
	\int_{E^+} |df|^2\ d\eta = \int_{E_f} |df|^2\ d\eta 
	\end{equation}
	The proof of Lemma \ref{passing to squares in graphings} is now  an exact replica of that of Lemma \ref{passing to squares in graphons}: the only extra point to note being that
	\begin{align}
	\begin{split}
	\int_{I^2}f(x)^2\ d\eta(x, y) &= \int_I \sum_y f(x)^2 \chi_E(x, y)\ d\mu(x) = \int_If(x)^2\deg(x)\ d\mu(x) = \norm{f}_v^2
	\end{split}
	\end{align}

We omit the details.
\end{proof}

The rest of the proof of Theorem \ref{cheeger-graphing} is quite similar to that of Theorem \ref{cheeger-graphon}. However, since this is one of the main theorems of this paper, we include the details for completeness.
The \hyperref[coarea formula for graphings]{Co-area Formula} Theorem \ref{coarea-graphing} gives:
\begin{equation}
	\int_{E_f} |df^2_+|\ d\eta = \int_0^\infty 2te_G(S_t^c, S_t)\ dt,
		\quad \text{and}\quad
	\int_{E_f} |df_-^2|\ d\eta = \int_{0}^\infty 2te_G(S_{-t}^c, S_{-t})\ dt
\end{equation}
But
\begin{align}
	\small
	\begin{split}
		\int_0^\infty 2te_G(S_t, S_t) \ dt &\geq h_G\int_0^\infty 2t\vol(S_t)\ dt = h_G\int_0^\infty 2t\lrb{\int_{S_{t}} \deg(x)\ d \mu(x)}\ dt\\
		&= h_G\int_0^\infty 2t\lrb{\int_I\chi_{S_{t}}(x) \deg(x)\ d \mu(x)}\ dt = h_G\int_0^\infty \int_I 2t\chi_{S_{t}}(x)\deg(x)\ d\mu(x)dt\\
		&= h_G\int_I\int_0^\infty 2t\chi_{S_{t}}(x)\deg(x)\ dtd\mu(x) = h_G\int_I\lrb{\int_0^\infty 2t\chi_{S_{t}}(x)\ dt}\deg(x)\ d\mu(x)\\ 
		&=h_G \int_{I}\lrb{ \int_0^{f_+(x)}2t\ dt}\deg(x) d\mu(x)\\
		&=h_G\int_{I} f_+^2(y)\deg(x) \ d\mu(x)
	\end{split}
\end{align}
Similarly
\begin{align}
	\small
	\begin{split}
		\int_0^\infty 2te_G(S_{-t}^c, S_{-t}) \ dt &\geq h_G\int_0^\infty 2t\vol(S_{-t}^c)\ dt = h_G\int_0^\infty 2t \lrb{\int_{S^c_{-t}}\deg(x)\ d\mu(x)} dt\\
		&= h_G\int_0^\infty 2t\lrb{\int_I\chi_{S_{-t}^c}(x)\deg(x)\ d\mu(x)} dt = h_G\int_0^\infty\int_I 2t\chi_{S^c_{-t}}(x)\deg(x)\ d\mu(x) dt\\
		&= h_G\int_I\int_0^\infty 2t\chi_{S^c_{-t}}(x)\deg(x)\ dtd\mu(x) = h_G\int_I\lrb{\int_0^\infty 2t\chi_{S^c_{-t}}(x)\ dt}\deg(x)d\mu(x)\\
		&= h_G\int_I\lrb{\int_0^{f_-(x)} 2t\ dt}\deg(x) d \mu(x)\\
		&= h_G\int_{I}f_-^2(x)\deg(x)\ d\mu(x, y)
	\end{split}
\end{align}
Therefore
\begin{align}
	\begin{split}
		\int_{E_f} |df^2_+|\ d\eta + \int_{E_f}|df^2_-|\ d\eta &\geq h_G \lrb{\int_{I} f_+^2(x)\deg(x) \ d\mu(x)+ \int_{I} f_-^2(x)\deg(x)\ d\mu(x)}\\
		&= h_G {\int_{I} (f_+^2(x)+ f_-^2(x))\deg(x) \ d\mu(x)}\\
		&= h_G \int_{I} f(x)^2\deg(x)\ d\mu(x)\\
		&= h_G \norm{f}^2_v
	\end{split}
\end{align}
Combining this with Equation \ref{passing to squares}, we have
\begin{equation}
	\int_{E_f} |df|^2\ d\eta \geq \frac{1}{8\norm{f}_v^2}  h^2_G\norm{f}_v^4
\end{equation}
This along with Equation \ref{passing to edges oriented by f}  gives
\begin{equation}
	\frac{\norm{df}_e^2}{\norm{f}_v^2} = \frac{\int_E |df|^2\ d\eta}{\norm{f}_v^2} \geq \frac{1}{8} h^2_G
\end{equation}
Lastly, using Equation \ref{passing from g to f} we have
\begin{equation}
	\norm{dg}_e^2\geq \frac{1}{8}h^2_G
\end{equation}
and Theorem \ref{cheeger-graphing} follows. \hfill $\Box$

\input{connected}

%
%
%

%
%
%
%
\bibliographystyle{alpha}
\bibliography{Cheeger.bib}

\end{document}

%% file: connected.tex

\section{Cheeger constant and connectedness}
\label{section:cheeger constant and connectedness}
\subsection{A connected graphon with zero Cheeger constant}\label{sec:connectedzerocheegergraphon}

Recall that a graphon is connected if for all measurable subsets $A$ of $I$ with $0< \mu_L(A)< 1$ we have $e_W(A, A^c) \neq 0$.
More generally, given a graphon $W$ and a measurable subset $S$ of $I$, we say that the restriction $W\vert_S$ is \define{connected} if for all measurable subsets $A$ of $S$ with $0< \mu_L(A)< \mu_L(S)$, we have
\begin{equation}
	\int_{A\times (S\setminus A)} W>0
\end{equation}

\begin{lemma}
	\label{van kampen theorem for connectedness}
	Let $W$ be a graphon and $S$ and $T$ be measurable subsets of $I$ such that $S\cup T=I$ and $S\cap T$ has positive measure.
 	Further assume that $W|_S$ and $W|_T$ are connected.
	Then $W$ is connected.
\end{lemma}
\begin{proof}
	Assume that $W$ is disconnected and let $A$ be a measurable subset of $I$ such that $0< \mu_L(A)< 1$ and $e_W(A, A^c)=0$.
	Then in particular we have
	\begin{equation}
		\int_{(A\cap S)\times (A^c\cap S)}W=0 \text{ and } \int_{(A\cap T)\times (A^c\cap T)}W=0
	\end{equation}
	The connectedness of $W|_S$ implies that either $A\cap S$ or $A^c\cap S$ has full measure in $S$.
	Without loss of generality assume that $A\cap S$ has full measure in $S$.
	Again, by the connectedness of $W|_T$ we have that $A\cap T$ or $A^c\cap T$ has full measure in $T$.
	In the former case we would have that $A$ has full measure in $I$ since $I=S\cup T$.
	In the latter case the measure of $S\cap T$ would be $0$.
	In any case we get a contradiction.
\end{proof}

\begin{example}
	\label{example:trivial leaf graphon}
	Consider the graphon $W$ given by the following figure.
	\begin{figure}[H]
		\centering
		\includegraphics{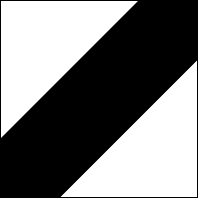}
				\caption{Example of a connected graphon.} 
		\label{trivial leaf graphon}
	\end{figure}
	\noindent

	The graphon takes the value $1$ at all the shaded points and $0$ at all other points.
	It follows by repeated use of Lemma \ref{van kampen theorem for connectedness}  that $W$ is connected.
\end{example}

Let $W$ be a graphon taking values in $\set{0, 1}$.
We call such a graphon a \define{neighborhood graphon} if $W^{-1}(1)$ contains an open neighborhood of the diagonal of the {\it open} square $(0,1) \times (0,1)$. Note that the graphon in Example \ref{example:trivial leaf graphon} contains instead an open neighborhood of the diagonal of the {\it closed} square $I^2$.

\begin{lemma}
	 \label{lemma:every neighborhood graphon is connected}
	 Every neighborhood graphon is connected.
\end{lemma}
\begin{proof}
	Let $W$ be a neighborhood graphon.
	For each $n>2$ let $S_n$ denote the interval $(1/n,  1-1/n)$.
	Each $W|_{S_n}$ is connected.
	This is because a copy of the graphon shown in Figure \ref{trivial leaf graphon} is embedded in the restriction of $W$ over $S_n\times S_n$.

	We argue by contradiction. Suppose that $W$ is disconnected.
	The there is $A\subseteq I$ with $0< \mu_L(A)< 1$ such that $e_W(A, A^c)=0$.
	Therefore for each $n$ we have $\int_{(A\cap S_n)\times (S_n\setminus A)}W=0$.
	By connectedness of $W|_{S_n}$, we must have that for any $n>2$, either $\mu_L(A\subseteq S_n)=\mu_L(S_n)$ or $\mu_L(A\cap S_n)=0$.
	If the former happens for some $n$, then it must happen for all $n$, and consequently $A$ is of full measure in $I$.
	The other possibility is that $\mu_L(A\cap S_n)=0$ for all $n$, but then $A$ has measure $0$.
	So in any case, we have a contradiction.
\end{proof}

\begin{example}
	\label{example:leaf graphon}
	A particular way of constructing a neighborhood graphon is the following.
	Let $f:I\to I$ be a continuous map such that $f(x)>x$ for all $0< x< 1$.
	Define a graphon $W_f$ as
	$$
		W_f(x, y)=
		\left\{
		\begin{array}{ll}
			1 & \text{if } x\leq y\leq f(x)\\
			0 & \text{if } f(x) < y\\
			W(y, x) &\text{if } x> y
		\end{array}
		\right.
	$$
	In other words, $W_f$ takes the value $1$ in the region trapped between the graph of $f$ and the reflection of the graph of $f$ about the $y=x$ line, and is $0$ everywhere else.
	For example, let $f(x)=\sqrt{x}$.
	Then the following diagram illustrates what $W_f$ looks like.
	\begin{figure}[H]
	\centering
		\includegraphics{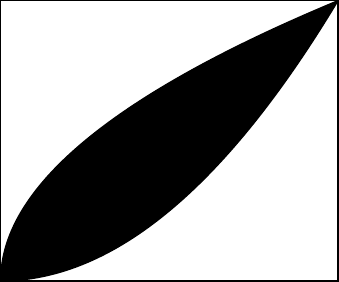}
		\caption{An example of a neighborhood graphon: $W_f$ corresponding to $f(x)=\sqrt{x}$.}
		\label{leaf of the square root}
	\end{figure}
	\noindent
	The graphon shown in Figure \ref{trivial leaf graphon} is also an example of a neighborhood graphon arising as $W_f$ for a suitably chosen continuous map $f:I\to I$.
\end{example}

\begin{example}
	\label{connected graphon with vanishing cheeger constant}
	Unlike in the case of finite graphs, the Cheeger constant of a connected graphon may be zero, as illustrated by Figure \ref{graphon with vanihing cheeger constant}.
	The graphon takes the value $1$ at all the shaded points, either gray or black, and the value $0$ at all the unshaded points.
	Call this graphon $W$.
	\begin{figure}[H]
		\centering
		\includegraphics{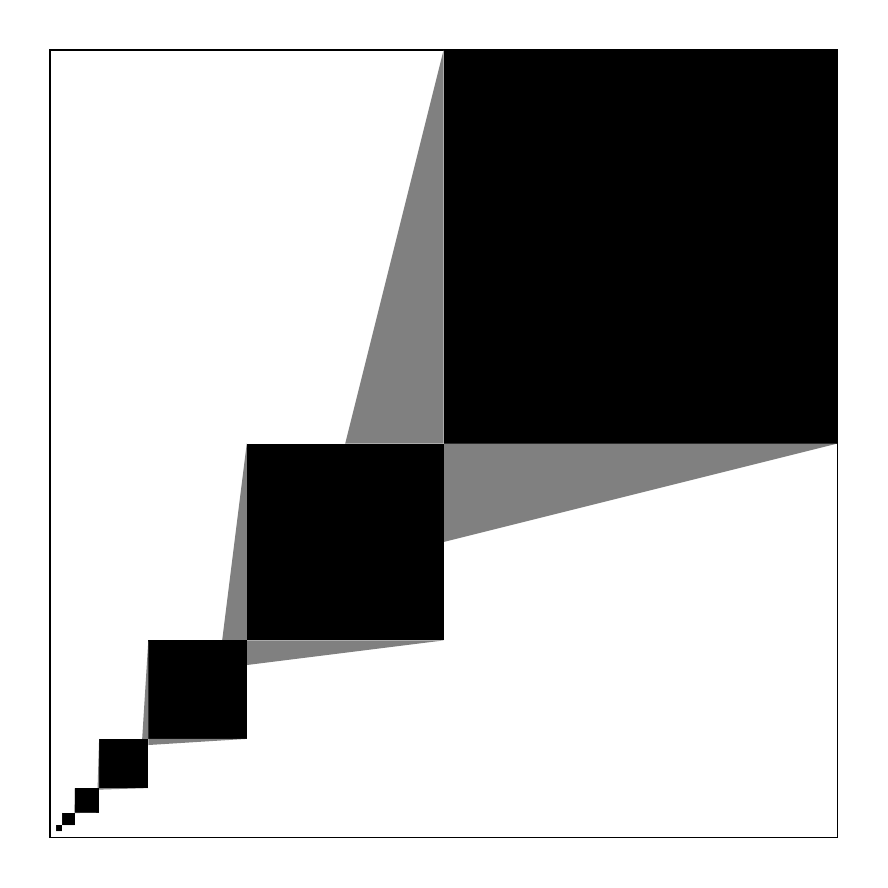}
		\caption{Example of a connected graphon whose Cheeger constant vanishes.} 
		\label{graphon with vanihing cheeger constant}
	\end{figure}

	The bottom left endpoints of the black squares in the above figure have coordinates $(1/2^n, 1/2^n)$, $n=1, 2, 3, \ldots$.
	The lengths of horizontal edges of the gray triangles above the $y=x$ line are $1/2^{2n+1}$, $n=1, 2, 3, \ldots$.
	Let $A_n$ be the interval $[0, 1/2^n]$.
	Then $e(A_n, A_n^c)$ equals
	 the measure of the gray region in $A^n\times A_n^c$.
	But there is only one gray triangle in this region.
	The sides of this right triangle (other than the hypotenuse) have lengths $1/2^{2n+1}$ and $1/2^n$.
	Thus $e(A_n, A_n^c)= 1/2^{3n+2}$.
	Let $V$ denote the total measure of the points shaded black.
	Then $\vol(A_n)\geq V/4^n + e(A_n, A_n^c)$ because the measure of the black region inside $A_n\times I$ is exactly $V/4^n$.
	For large $n$ we have $\vol(A_n)$ is at most half the total volume.
	Thus for large $n$ we have
	\begin{equation}
	h_W(A_n)= \frac{e(A_n, A_n^c)}{\vol(A_n)}\leq \frac{1/2^{3n+2}}{V/4^n + 1/2^{3n+2}} = \frac{1/2^{n+2}}{V+1/2^{n+2}}
	\end{equation}
	This is zero in the limit and thus the Cheeger constant of this graphon is zero.
	This graphon is connected by Lemma \ref{lemma:every neighborhood graphon is connected} because it is a neighborhood graphon.
\end{example}

\subsection{A connected graphing with zero Cheeger constant}
\label{subsection:connected graphing with vanishing cheeger constant} We prove in this section that the {\it irrational cyclic graphing} \cite[Example 18.17]{lovasz_large_networks} has zero Cheeger constant.

Let $a_0$ be an irrational number.
We get a bounded degree Borel graph $(I, E)$ on $I$ by joining two points $x$ and $y$ if $|x-y|=a_0$.
The triple $(I, E, \mu_L)$ then becomes a graphing (recall that $\mu_L$ denotes the Lebesgue measure).

An equivalent way of thinking of this graphing is as follows:
Let $T:S^1\to S^1$ be the rotation of the unit circle by an angle which is an irrational multiple of $2\pi$.
We get a Borel graph on $S^1$ by declaring $(x, y)\in S^1\times S^1$ to be an edge if and only if $T(x)=y$ or $T^{-1}(x)=y$.
Equipping the circle with the Haar measure $\mu_H$, this Borel graph is in fact a graphing \cite[Example 18.17]{lovasz_large_networks}.
We will denote this graphing by $G$.
This is a connected graphing because if $A$ is a measurable subset of $S^1$ such that $e_G(A, A^c)=0$, then we would have that $T^{-1}(A)\cap A^c$ has measure $0$
and hence $A$ is $T-$invariant. By  ergodicity of the action of $T$ on $(S^1,\mu_H)$, we infer that $A$ is either of  zero or  full measure.

We show that the Cheeger constant of this graphing is zero.
Let $X$ be a small arc of the circle with one end-point  at $(1, 0)$.
\begin{figure}[H]
	 \centering
	 \includegraphics[scale=1.5]{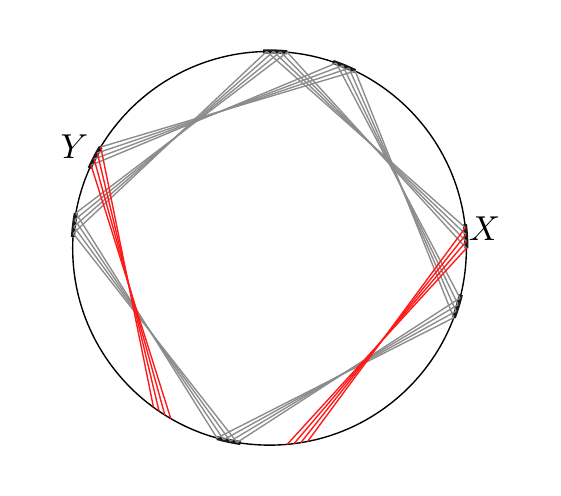}
	 \caption{Graphing corresponding to an irrational rotation of the circle.}
\end{figure}
\noindent
Given  $N > 0$, we can choose $X$ small enough   so that\\

\begin{enumerate}
\item $T^i(X) \cap T^j(X) = \emptyset$, for $0\leq i < j \leq N$,\\

\item For $A:=X\sqcup T(X)\sqcup\cdots\sqcup T^N(X)$, $\mu_H(A) \leq \frac{1}{2}$.\\

\end{enumerate}

Write $Y=T^N(X)$.
The only edges that contribute to $e_G(A, A^c)$ are the ones going from $Y$ to $T(Y)$ and the ones going from $T^{-1}(X)$ to $X$.
These are shown in red in the above figure.
Thus $e_G(A, A^c)\leq 2\mu_H(X)$.
Therefore we have
\begin{equation}
	h_G\leq
	h_G(A)
	\leq \frac{e_G(A, A^c)}{\vol_G(A)}
	\leq 2\frac{\mu_H(A)}{(N+1)\mu_H(A)} = 2/(N+1)
\end{equation}
Since $N$ is arbitrary, we conclude that $h_G=0$.

\begin{rmk}
An example of a connected graphing $G$ with positive $\lambda(G)$ and hence (by Theorem \ref{cheeger-graphing}) positive Cheeger constant $h_G$ has been described by Lovasz in \cite[Example 21.5]{lovasz_large_networks} under the rubric of `expander graphings'. 
\end{rmk}
\subsection{A necessary and sufficient condition for connectedness of a graphon}

In the special case that a graphon $W$ has degree of every vertex uniformly bounded below, we shall now proceed to give  a necessary and sufficient condition in terms of the Cheeger constant for $W$ to be connected. This is analogous to the statement that a finite graph is connected if and only if its Cheeger constant is positive.

\begin{prop}\label{connectedpositivecheeger}
	 Let $\varepsilon>0$ and $W$ be a graphon such that the degree $d_W(x) \geq \varepsilon$ for all $x \in I$.
	Then $W$ is connected if and only if $h_W > 0$.
\end{prop}

We provide two proofs of the above result.
The first is an application  of Theorem \ref{buser-graphon} the Buser inequality for graphons  and is essentially self-contained using some basic facts about compact operators. The second proof uses a structural lemma about connected graphons proved in \cite{bbcr2010}. 

\begin{definition}{\cite[pg. 196]{limaye_functional_analysis}}
	We say that $\lambda\in \R$ is an \define{approximate eigenvalue} of a bounded linear operator $T:H\to H$ of a Hilbert space $H$ if the image of $T-\lambda I$ is not bounded below.
\end{definition}

\begin{lemma}
	\label{rayleigh quotient is an approximate eigenvalue}
	\cite[Lemma 27.5(a)]{limaye_functional_analysis}
	For any bounded linear self-adjoint operator $T:H\to H$ on a Hilbert space $H$, we have that
\begin{equation}
		\inf_{x\in H:\ \norm{x}=1} \ab{Tx, x}
\end{equation}
is an approximate eigenvalue of $T$.
\end{lemma}

\begin{lemma}
	 \label{approximate eigenvalue of a compact operator is an eigenvalue}
	 \cite[Lemma 28.4]{limaye_functional_analysis}
	 If $T:H\to H$ is a compact operator then every approximate eigenvalue of $T$ is an actual eigenvalue of $T$.
\end{lemma}

\begin{lemma}
	\label{inverse diagonal times adjacency operator is compact}
	Let $W$ be a connected graphon with $d_W$ bounded below by a positive real.
	Then the map $(1/d_W)T_W:L^2(I, \nu)\to L^2(I, \nu)$ is a compact operator.
\end{lemma}

\begin{proof}
	Since $d_W$ is  bounded below, it follows that $L^2(I, \mu_L)$ and $L^2(I, \nu)$ have comparable norms. Therefore $I:L^2(I, \mu_L)\to L^2(I, \nu)$ is a bounded linear isomorphism.
 	The operator $T_W:L^2(I, \mu_L)\to L^2(I, \mu_L)$ is compact \cite[Section 7.5]{lovasz_large_networks}.
 Since $d_W$ is bounded below, i.e.\  $1/d_W$ is $L^\infty$, it follows that the operator $(1/d_W)T_W:L^2(I, \mu_L)\to L^2(I, \mu_L)$ is also compact.

 Take any bounded sequence $(f_n)$ in $L^2(I, \nu)$.
 Then $(f_n)$ is bounded in $L^2(I, \mu_L)$ too because of comparability of norms.
  Since $(1/d_W)T_W:L^2(I, \mu_L)\to L^2(I, \mu_L)$ is compact, there exists a subsequence $(f_{n_k})$ such that $(1/d_W)(T_Wf_{n_k})$ converges in $L^2(I, \mu_L)$.
 Again, the comparability of norms give that $(1/d_W)(T_W f_{n_k})$ converges in $L^2(I, \nu)$ as required.
\end{proof}

Note that for any graphon $W$, the Laplacian $\Delta_W:L^2(I, \nu)\to L^2(I, \nu)$ restricts to  a linear operator $\Delta_W:\constone^{\perp}_v\to \constone^\perp_v$.\\

\begin{proof} of Proposition \ref{connectedpositivecheeger}:\\
	 If $h_W>0$ then clearly $W$ is connected.
	 So we need to prove the other direction.
	 Let $W$ be a connected graphon
	  with $d_W(x)\geq \varepsilon$ for all $x\in I$.
	 Lemma \ref{inverse diagonal times adjacency operator is compact} ensures that $(1/d_W)T_W$ is a compact operator on $L^2(I, \nu)$ and it is easy to check that it restricts to a linear operator from $\constone^\perp_v$ to itself. 
	Throughout we will think of $\Delta_W$ and $(1/d_W)T_W$ as linear operators in $\constone^\perp_v$.
	Now by Lemma \ref{rayleigh quotient approximation} $\lambda(W)$ is an approximate eigenvalue of $\Delta_W$.
	Thus the image of $\Delta_W-\lambda(W)I = (1-\lambda(W))I-(1/d_W) T_W$ is not bounded below in $\constone^\perp_v$.
	Hence $1-\lambda(W)$ is an approximate eigenvalue of $(1/d_W)T_W$.
	But $(1/d_W)T_W:\constone^\perp_v\to \constone^\perp_v$ is a compact operator, and thus by Lemma \ref{approximate eigenvalue of a compact operator is an eigenvalue} we have that $1-\lambda(W)$ is in fact an eigenvalue of $(1/d_W)T_W$.
	Therefore $\lambda(W)$ is an eigenvalue of $\Delta_W$.
	Let $f\in \constone^\perp_v$ be nonzero such that $\Delta_Wf=\lambda(W)f$.
	If $h_W$ were equal to $0$, then by the Buser inequality (Theorem \ref{buser-graphon}) we have that $\lambda(W)=0$.
	Thus $\Delta_Wf=0$, which is equivalent to saying that $df=0$.
	But as observed in the first paragraph of Section \ref{sec:buser-graphon} we then have that $f$ is a constant function and hence the only way it can belong to $\constone^\perp_v$ is that $f=0$--a contradiction.
\end{proof}

Now we give the second proof of Proposition \ref{connectedpositivecheeger}.


\begin{defn}
	 \cite{bbcr2010}
	 Let $W$ be a graphon and $0< a, b< 1$ be real numbers.
	 An $(a, b)$-\define{cut} in $W$ is a partition $\set{A, A^c}$ of $I$ with $a<\mu_L(A)< 1-a$ such that $e_W(A, A^c)\leq b$.
\end{defn}
Note that a graphon is connected if and only if it admits no $(a, 0)$-cut for any $0< a\leq 1/2$.

\begin{lemma}
	 \label{connected graphons does not admit an ab cut}
	 \cite[Lemma 7]{bbcr2010}
	 Let $W$ be a connected graphon and $0< a< 1/2$.
	 Then there is some $b>0$ such that $W$ admits no $(a, b)$-cut.
	 \vspace{1em}
\end{lemma}

{\bf Alternate proof} of Proposition \ref{connectedpositivecheeger} using Lemma \ref{connected graphons does not admit an ab cut}:\\
	Let $W$ be a graphon with $d_W(x)\geq \varepsilon$ for all $x$.
	We prove the non-trivial direction.
	Assume $W$ is connected.
	We will show that $h_W>0$.
	Assume on the contrary that $h_W=0$.
	Then for each $n\geq 1$ there is a measurable subset $A_n$ of $I$ with $0<\mu_L(A_n)< 1$ such that $h_W(A_n)< 1/n$.
	After passing to a subsequence, there are two cases to consider.

	Case 1: $\mu_L(A_n)\to 0$ as $n\to \infty$.

	In this case we have for large $n$ that
	\begin{align}
		 \begin{split}
		 	h_W(A_n) &= \frac{e(A_n, A_n^c)}{\vol(A_n)} = \frac{\vol(A_n) - \int_{A_n\times A_n} W}{\vol(A_n)}\\
			&= 1 - \frac{\int_{A_n\times A_n} W}{\vol(A_n)} \geq 1- \frac{\mu_L(A_n)^2}{\varepsilon \mu_L(A_n)}\\
			&\geq 1-\mu_L(A_n)/\varepsilon
		 \end{split}
	\end{align}
	But this contradicts the assumption that $h_W(A_n)< 1/n$ for all $n$.

	Case 2: $\mu_L(A_n)\to t$ for some $t>0$.

	Let $0<a< 1/2$ be such that $a< t< 1-a$.
	Now by Lemma \ref{connected graphons does not admit an ab cut}, there is a $b>0$ such that $W$ admits no $(a, b)$-cut.
	Therefore for all $n$ large enough we have
	\begin{align}
		 \begin{split}
			  h_W(A_n) &= \frac{e(A_n, A_n^c)}{\min\set{\vol(A_n), \vol(A_n^c)}}\\
						&\geq \frac{b}{\varepsilon \min\set{ \mu_L(A_n), \mu_L(A_n^c)}}\geq \frac{b}{\varepsilon (1-a)}
		 \end{split}
	\end{align}
	which again contradicts the assumption that $h_W(A_n)< 1/n$ for all $n$.
\hfill $\blacksquare$

%
%
%

%
